\documentclass[12pt]{article}
\usepackage{times}
\usepackage{amsmath,amsfonts,amstext,amssymb,amsbsy,amsopn,amsthm,eucal,slashed,amstext,enumerate}

\usepackage{color}
\usepackage{txfonts}
\usepackage{dsfont}
\usepackage{graphicx}   

\numberwithin{equation}{section}
  \theoremstyle{plain}
 \newtheorem{theorem}[equation]{Theorem}
 \newtheorem{thm}[equation]{Theorem}
\newtheorem{proposition}[equation]{Proposition}
 \newtheorem{lemma}[equation]{Lemma}

 \theoremstyle{remark}

 \newtheorem{remark}[equation]{Remark}

\theoremstyle{definition}
 \newtheorem{definition}[equation]{Definition}

\newcommand{\abs}[1]{\lvert#1\rvert}
\newcommand{\norm}[1]{\lvert\lvert#1\rvert\rvert}
\newcommand{\ip}[1]{\left\langle#1\right\rangle}

\newcommand{\Rc}{{\rm Ric}}
\newcommand{\Ric}{{\rm Ric}}

\newcommand{\vol}{{\rm vol}}

\newcommand{\grad}{{\rm grad}}

\newcommand{\dR}{\mathds{R}}

\newcommand{\Hess}{{\rm Hess}}

\newcommand{\R}{{ \rm R}}
\newcommand{\p}{\parallel}
\newcommand{\Rm}{\text{Rm}}

\newcommand{\eps}{\varepsilon}

\newcommand{\cH}{\mathcal{H}}

\newcommand{\M}{\mathcal{M}}
\newcommand{\F}{\mathcal{F}}

\newcommand{\Lap}{\Delta}
\newcommand{\D}{\nabla}

\newcommand{\RR}{\mathbb{R}}

\begin{document}

\title{Weak solutions for the Ricci flow I}
\author{Robert Haslhofer and Aaron Naber\thanks{R.H. has been supported by NSF grant DMS-1406394, A.N. has been supported by NSF grant DMS-1406259.}}
\date{\today}

\maketitle
\begin{abstract}
This is the first of a series of papers, where we introduce a new class of estimates for the Ricci flow, and use them both to characterize solutions of the Ricci flow and to provide a notion of weak solutions to the Ricci flow in the nonsmooth setting.  In this first paper, we prove various new estimates for the Ricci flow, and show that they in fact characterize solutions of the Ricci flow.  Namely, given a family $(M,g_t)_{t\in I}$ of Riemannian manifolds, we consider the path space $P\mathcal{M}$ of its space time $\mathcal{M}=M\times I$.
Our first characterization says that $(M,g_t)_{t\in I}$ evolves by Ricci flow if and only if an infinite dimensional gradient estimate holds for all functions on $P\mathcal{M}$.
We prove additional characterizations in terms of the $C^{1/2}$-regularity of martingales on path space, as well as characterizations in terms of log-Sobolev and spectral gap inequalities for a family of Ornstein-Uhlenbeck type operators. Our estimates are infinite dimensional generalizations of much more elementary estimates for the linear heat equation on $(M,g_t)_{t\in I}$, which themselves generalize the Bakry-Emery-Ledoux estimates for spaces with lower Ricci curvature bounds.
Based on our characterizations we can define a notion of weak solutions for the Ricci flow. We will develop the structure theory of these weak solutions in subsequent papers. 
\end{abstract}

\small{\tableofcontents}


\section{Introduction}


\subsection{Background and overview}

The Ricci flow, introduced by Richard Hamilton \cite{Ham82}, evolves Riemannian manifolds in time and is given by the equation
\begin{equation}\label{eq_ricci_flow}
 \partial_t g_t=-2\Ric_{g_t}.
\end{equation}
As with all geometric equations, the key to the analysis of \eqref{eq_ricci_flow} is to prove estimates that are strong enough to capture the analytic and geometric behavior. 
Many of the known estimates for the Ricci flow are similar in nature to -- but often have been harder to develop than -- the corresponding estimates for other geometric equations.  Since the geometry itself is evolving, even the most basic geometric quantities, like the heat kernel, can behave quite badly. Furthermore, many techniques from geometric analysis that rely on the presence of an ambient space (or a fixed underlying manifold) are not available for the Ricci flow. In particular, it has been a longstanding open problem to find a notion of weak solutions for the Ricci flow.\\

The goal of this paper, the first in a series, is to introduce a new class of estimates for the Ricci flow.  Our new estimates not only give new information about solutions of the Ricci flow, but are designed to be sufficiently powerful that they give analytic criteria for determining when a family of Riemannian manifolds solves the Ricci flow.  That is, we will see that if a family $(M,g_t)_{t\in I}$ of Riemannian manifolds satisfies the analytic estimates of this paper, then in fact this family solves \eqref{eq_ricci_flow}.  Such analytic criteria can be used to define weak solutions and have become of increasing importance in other areas of Ricci curvature, see for instance \cite{LottVillani,Sturm,AGS,Naber}, but have not been available up to now for the Ricci flow itself.\\

We start with the comparably simple task of characterizing supersoluions of the Ricci flow, i.e. families $(M,g_t)_{t\in I}$ such that $\partial_t g_t\geq -2\Ric_{g_t}$,
see Section \ref{sec_introsuper} and Section \ref{sec_supersol}. As summarized in Theorem \ref{thm_supersol}, supersoluions can be characterized in terms of various estimates for the linear heat equation on $(M,g_t)_{t\in I}$.  These estimates generalize the Bakry-Emery-Ledoux estimates  for manifolds with lower Ricci curvature bounds \cite{BakryEmery,BakryLedoux}, see also McCann-Topping \cite{McCannTopping}.  In particular, one can characterize supersolutions in terms of a log-Sobolev inequality, and a Poincare-inequality.  The log-Sobolev inequality is not the one discovered by Perelman \cite{Per1}, but the more recent one from Hein-Naber \cite{HN_eps}.\\

To characterize solutions of the Ricci flow, and not just supersolutions, we prove infinite dimensional generalizations of the above estimates.
Motivated by work in stochastic analysis \cite{Malliavin_icm,Driver_ibp,Fang,AE_logsob,Hsu_logsob} and prior work of the second author \cite{Naber}, our approach to finding such infinite dimensional generalizations is to do analysis on path space. More precisely, it turns out that the right path space to consider, is the space $P\mathcal{M}$ of continuous curves in the space-time $\mathcal{M}=M\times I$,
which are allowed to move arbitrarily along the manifold $M$ but are required to move backwards along the $I$ factor with unit speed.
To be able to do analysis on $P\mathcal{M}$ we have to set up quite a bit of machinery from stochastic analysis,
notably the notions of Wiener measure, stochastic parallel transport, parallel gradient and Malliavin gradient, adapted to our space-time setting.
We describe this briefly in Section \ref{sec_introprelim} and give a comprehensive treatment in Section \ref{sec_prelim}.
For example, the construction of parallel transport is quite subtle, since almost no curve of Brownian motion is $C^1$. Nevertheless, using ideas from Eells-Elworthy-Malliavin \cite{Elworthy,Malliavin}, we can make sense of parallel transport on space-time for almost every curve of Brownian motion, see Section \ref{sec_brownian}.\\

Having set the stage, let us now discuss our infinite dimensional estimates.
Our first characterization in Section \ref{sec_introgradient} directly relates solutions of the Ricci flow to gradient estimates on path space.
Specifically, we will see that a family $(M,g_t)_{t\in I}$ evolves by Ricci flow if and only if a certain gradient inequality (R2) holds for all functions on $P\mathcal{M}$.  We will see how this directly generalizes the gradient estimate (S2) proved in Theorem \ref{thm_supersol} for supersolutions.
Our second characterization in Section \ref{sec_intromartingales} is in terms of the time regularity of martingales on path space.  Specifically, we will see that martingales $F^\tau$ on path space satisfy a precise $C^{1/2}$-H\"older estimate (R3) if and only if the family $(M,g_t)_{t\in I}$ evolves by Ricci flow.  
Our third characterization in Section \ref{sec_intrologsob} is in terms of an infinite dimensional log-Sobolev inequality (R4), and our final characterization in Section \ref{sec_introgap} is in terms of the corresponding spectral gap (R5).
Our characterizations of solutions of the Ricci flow can be thought of as infinite dimensional generalizations of the estimates for supersolutions.
Namely, if we evaluate our infinite dimensional estimates for the simplest possible test functions, i.e. functions on path space that only depend on the value of the curve at a single time, then we actually recover the finite dimensional estimates from Theorem \ref{thm_supersol}. Of course, there are many more sophisticated test functions that we can plug in our estimates, and this is one of the reasons why our estimates are actually strong enough to characterize solutions, and not just supersolutions.
Our characterizations of solutions of the Ricci flow constitute the main results of this article and are summarized in Theorem \ref{main_thm_intro}.\\

Let us also emphasize that Theorem \ref{main_thm_intro} truly relies on ideas from stochastic analysis, i.e. doing analysis on path space $P\M$, as it seems that analysis on $(M,g_t)_{t\in I}$ can only be used to characterize supersolutions but not solutions. In fact, some indications that stochastic analysis might be useful in the study of Ricci flow have already appeared previously in the literature: Arnoundon-Coulibaly-Thalmaier proved the existence of Brownian motion in a time dependent setting \cite{ACT} (see also \cite{Coulibaly}), and used this to prove a Bismut type formula for the Ricci flow. Kuwada-Philipowski studied the relationship between time dependent Brownian motion and Perelman's $\mathcal{L}$-geodesics and obtained a nice nonexplosion result \cite{KuwadaPhilip,KuwadaPhilip2}  (see also \cite{Cheng}), and Guo-Philipowski-Thalmaier found some applications of stochastic analysis to ancient solutions \cite{GPT}. Based on our new estimates there are many more directions to explore.\\

In future papers of this series we will use our estimates to investigate singularities in the Ricci flow.  In most situations, the Ricci flow develops singularities in finite time. Typically, the curvature blows up in certain regions but remains bounded on the remaining parts of the manifold \cite{Hamilton_survey}.  One would then like to understand these singularities and find ways to continue the flow beyond the first singular time.\\

The formation of singularities is of course an ubiquitous phenomenon in the study of nonlinear PDEs.
For other geometric evolution equations there are good notions of weak solutions that allow one to continue the flow through any singularity, e.g. Brakke and level set solutions for the mean curvature flow \cite{Brakke,EvansSpruck,CGG}, and Chen-Struwe solutions for the harmonic map heat flow \cite{ChenStruwe}. For the Ricci flow however, it is only known in a few special - albeit very important - cases, how to continue the flow through singularities. Most notably, Perelman's Ricci flow with surgery \cite{Per1,Per2} provides a highly successful way to deal with the formation of singularities in dimension three. Surgery has also been implemented in the case of four-manifolds with positive isotropic curvature \cite{Ham_pic,CZ_pic}.
Recently, Kleiner-Lott proved the beautiful result that as the surgery parameters degenerate it is possible to pass to certain limits, called singular Ricci flows \cite{KL_singular}. 
Also, there has been a lot of progress in the K\"ahler case, see e.g. Song-Tian \cite{SongTian} and Eyssidieux-Guedj-Zeriahi \cite{EGZ_kahler}. In most other cases however, it is a widely open problem how to deal with the formation of singularities.\\

In the second paper of this series we will use the estimates of this first paper to give a notion of the Ricci flow for a family of metric-measure spaces.  Using analytic characterizations to define weak solutions is a well developed tool in the context of lower Ricci curvature \cite{LottVillani,Sturm,AGS}, and more recently in the context of bounded Ricci curvature \cite{Naber}.  Similarly, based on the characterizations of Theorem \ref{main_thm_intro} we will define a notion of weak solutions for the Ricci flow and develop their theory. We will discuss this in subsequent papers, but let us briefly describe the idea. We consider metric-measure spaces $\mathcal{M}$ equipped with a time function and a linear heat flow. We call $\mathcal{M}$ a weak solution of the Ricci flow if and only if the gradient estimate (R2) holds on $P\mathcal{M}$. We then establish various geometric and analytic estimates for these weak solutions.
One of our applications concerns a question of Perelman about limits of Ricci flows with surgery \cite{Per1}. Namely, the metric completion of the space-time of Kleiner-Lott \cite{KL_singular}, which they obtained as a limit of Ricci flows with surgery where the neck radius is sent to zero, is a weak solution in our sense.

\subsection{Characterization of supersolutions of the Ricci flow}\label{sec_introsuper}

As a motivation for our approach to characterize solutions of the Ricci flow, let us first characterize supersolutions, i.e. smooth families of Riemannian manifolds such that
\begin{equation}
\partial_t g_t\geq -2\Rc_{g_t}.
\end{equation}

To fix notation, let $(M,g_t)_{t\in I}$ be a smooth family of Riemannian manifolds, where $I=[0,T_1]$. 
To avoid technicalities, we assume throughout the paper that all manifolds are complete and that
\begin{equation}\label{tech_ass}
\sup_{M\times I}\left(\abs{\Rm}+\abs{\partial_t g_t}+\abs{\nabla \partial_t g_t}\right)<\infty.
\end{equation}
However, all our estimates are independent of the implicit constant in \eqref{tech_ass}.
We consider the heat equation $(\partial_t-\Lap_{g_t})w=0$ on our evolving manifolds $(M,g_t)_{t\in I}$.
For every $s,T\in I$ with $ s\leq T$, and every smooth function $u$ with compact support, we write $P_{sT}u$ for the solution at time $T$ with initial condition $u$ at time $s$.
In other words,
\begin{equation}
(P_{sT}u)(x)=\int_M u(y)\, H(x,T\,|\,y,s) d\vol_{g(s)}(y),
\end{equation}
where $H(x,T\,|\,y,s)$ is the heat kernel with pole at $(y,s)$. We write $d \nu_{(x,T)}(y,s)=H(x,T\,|\,y,s) d\vol_{g(s)}(y)$. It is often useful to think of $d \nu_{(x,T)}$ as the adjoint heat kernel measure based at $(x,T)$.

The following theorem summarizes our characterizations of supersolutions of the Ricci flow.

\begin{thm}[Characterization of supersolutions of the Ricci flow]\label{thm_supersol}
For every smooth family $(M,g_t)_{t\in I}$ of Riemannian manifolds (complete, satisfying \eqref{tech_ass}), the following conditions are equivalent:
\begin{enumerate}[(S1)]
 \item \label{S1} The family $(M,g_t)_{t\in I}$ is a supersolution of the Ricci flow,
 \begin{equation*}\partial_t g_t\geq -2\Rc_{g_t}.
 \end{equation*}
\item \label{S2} For all test functions $u$, the heat equation on $(M,g_t)_{t\in I}$ satisfies the gradient estimate
\begin{equation*}
\abs{\D P_{sT} u}\leq P_{sT}\abs{\D u}.
\end{equation*}
\item \label{S3} For all test functions $u$,  the heat equation on $(M,g_t)_{t\in I}$ satisfies the estimate
\begin{equation*}
\abs{\D P_{sT} u}^2\leq P_{sT}\abs{\D u}^2.
\end{equation*}
\item \label{S4} For all functions $u:M\to \dR$ with $\int_M u^2(y) \, d\nu_{(x,T)}(y,s)=1$, we have the log-Sobolev inequality
\begin{equation*}
 \int_M u^2(y)\log u^2(y) \, d\nu_{(x,T)}(y,s)\leq 4(T-s)\!\! \int_M\abs{\D u}_{g_s}^2\! (y)\,d\nu_{(x,T)}(y,s).
\end{equation*}
\item \label{S5}
For all functions $u:M\to \dR$ with $\int_M u(y) \, d\nu_{(x,T)}(y,s)=0$, we have the Poincare-inequality
\begin{equation*}
\int_M u^2 \, d\nu_{(x,T)}(y,s)\leq 2 (T-s)\!\! \int_M \abs{\D u}^2_{g_s}\, d\nu_{(x,T)}(y,s).
\end{equation*}
\end{enumerate}
\end{thm}

In essence, this all follows from the Bochner-formula for the heat operator $\Box=\partial_t-\Lap_{g_t}$,
\begin{equation}\label{parboch}
 \Box \abs{\D u}^2=2\ip{\D u,\D\Box u}-2\abs{\D^2 u}^2-(\partial_t g+2\Rc)(\grad u,\grad u),
\end{equation}
see Section \ref{sec_supersol} for the (easy) proof of Theorem \ref{thm_supersol}. The reader can also view this as a good toy model for the more sophisticated infinite-dimensional computations on path space that we carry out in later sections.

\begin{remark}
 Theorem \ref{thm_supersol} can be thought of as parabolic version of the Bakry-Emery characterization of nonnegative Ricci curvature \cite{BakryEmery,BakryLedoux}.
Another interesting characterization of supersolutions of the Ricci flow, in terms of the Wasserstein distance, has been given by McCann-Topping \cite{McCannTopping}.
\end{remark}

\subsection{Characterization of solutions of the Ricci flow}\label{sec_introsol}

In this section we describe our main estimates on path space, and use them to characterize solutions of the Ricci flow.

\subsubsection{Stochastic analysis on evolving manifolds}\label{sec_introprelim}

Our estimates require quite some machinery from stochastic analysis,
notably the notions of Wiener measure, stochastic parallel transport, parallel gradient and Malliavin gradient, adapted to our time-dependent setting.
We will now briefly describe these notions, and refer to Section \ref{sec_prelim} for a more complete treatment.\\

Let $(M,g_t)_{t\in I}$ be a smooth family of Riemannian manifolds, where $I=[0,T_1]$. We recall that we always assume that our manifolds are complete and that \eqref{tech_ass} is satisfied, though the second assumption is for convenience.
Throughout this work we will think of the evolving manifolds in terms of the \textbf{space-time} $\M=M\times I$.
As observed by Hamilton \cite{Hamilton_Harnack} there is a
natural \textbf{space-time connection} defined by
\begin{equation}\label{eq_spacetime_connection}
\nabla_X Y=\nabla^{g_t}_X Y,\qquad\qquad \nabla_t Y=\partial_t Y+\frac{1}{2}\partial_t g_t(Y,\cdot)^{\sharp_{g_t}}.
\end{equation}
The point is that this connection is compatible with the metric, i.e. $\tfrac{d}{dt}\abs{Y}_{g_t}^2=2\langle Y,\nabla_t Y\rangle$.

It is useful to consider \textbf{space-time curves} going \emph{backwards in time}, c.f. \cite{LiYau,Per1}. Namely, for each $(x,T)\in \M$, we consider the \textbf{based path space} $P_{(x,T)}\M$ consisting of all space-time curves of the form $\{\gamma_\tau=(x_\tau,T-\tau)\}_{\tau\in [0,T]}$,
where $\{x_\tau\}_{\tau\in [0,T]}$ is a continuous curve in $M$ with $x_0=x$.

We equip the path space $P_{(x,T)}\M$ with a probability measure $\Gamma_{(x,T)}$, that we call the \textbf{Wiener measure} of \textbf{Brownian motion} on our evolving family of manifolds, based at $(x,T)$. The measure $\Gamma_{(x,T)}$ is uniquely characterized by the following property. If $e_{\bf{\sigma}}:P_{(x,T)}\M\to M^{k}$, $\gamma\mapsto (x_{\sigma_1},\ldots, x_{\sigma_k})$, is the \textbf{evaluation map} at ${\bf{\sigma}}=\{0\leq \sigma_1\leq\ldots\leq \sigma_k\leq T\}$, and if we write $s_i=T-\sigma_i$, then 
\begin{equation}\label{intro_pushforwardwiener}
e_{\bf{\sigma},\ast}d\Gamma_{(x,T)}(y_1,\ldots,y_k)=H(x,T|y_1,s_1)d\vol_{g_{s_1}}(y_1)\cdots H(y_{k-1},s_{k-1}|y_k,s_k)d\vol_{g_{s_k}}(y_k),
\end{equation}
where $H$ is the heat kernel of $\partial_t-\Lap_{g_t}$; see Section \ref{sec_brownian} for the construction of Brownian motion.

It is often convenient to consider the \textbf{total path space} $P_T\M=\cup_{x\in M} P_{(x,T)}\M$. Note that we can identify $(P_T\M,\Gamma_{(x,T)})$ with $(P_{(x,T)}\M,\Gamma_{(x,T)})$,
since the measure $\Gamma_{(x,T)}$ concentrates on curves starting at $(x,T)$. Sometimes it is also useful to equip the total path space  $P_T\M$ with the measure $\Gamma_T=\int \Gamma_{(x,T)} d\vol_{g_T}(x)$.\\

The space $(P_T\M,\Gamma_{(x,T)})$ can be equipped with a notion of \textbf{stochastic parallel transport}, a family of isometries $P_\tau(\gamma):(T_{x_\tau}M,g_{T-\tau})\to (T_xM,g_T)$.
If the curves $\gamma$ were $C^1$, then $P_\tau(\gamma)$ would just be the parallel transport from differential geometry, with respect to the natural space-time connection defined in \eqref{eq_spacetime_connection}. Of course, almost no curve of Brownian motion is $C^1$. Nevertheless, using deep ideas from Eells-Elworthy-Malliavin we can still make sense of $P_\tau(\gamma)$ for almost every curve $\gamma$,
see Section \ref{sec_brownian} for the construction.\\

The space $(P_T\M,\Gamma_{(x,T)})$ can be equipped with two natural notions of gradient.
Suppose first that $F:P_{(x,T)}\M\to \dR$ is a \textbf{cylinder function},
i.e. a function of the form $F=u\circ e_\sigma$, where $e_\sigma: P_{(x,T)}\M\to M^k$ is an evaluation map and $u:M^k\to \dR$ is a smooth function with compact support.
If $v\in (T_xM,g_T)$, then for almost every (a.e.) curve $\gamma$, we can consider the continuous vector field $V=\{V_\tau=P_\tau^{-1}v\}_{\tau\in [0,T]}$ along $\gamma$, where $P_\tau=P_\tau(\gamma)$ denotes stochastic parallel transport as in the previous paragraph.
Note that the directional derivative $D_VF(\gamma)$ is well defined, as a limit of difference quotients as usual.

The \textbf{parallel gradient} $\D^\p F(\gamma)\in (T_xM, g_T)$ is then defined by the condition that
\begin{equation}\label{intro_pargrad}
 D_V F(\gamma)=\langle \D^\p F(\gamma),v\rangle_{(T_xM, g_T)}
\end{equation}
for all $v\in (T_xM,g_T)$, where $V=\{V_\tau=P_\tau^{-1}v\}_{\tau\in [0,T]}$ is the parallel vector field associated to $v$, as above. More generally, there is a one parameter family of parallel gradients $\D^\p_\sigma$ ($0\leq \sigma\leq T$), which captures the part of the gradient coming from the time interval $[\sigma,T]$. In particular, $\D^\p=\D^\p_0$.

The \textbf{Malliavin gradient} $\D^\cH F$ is defined along similar lines, but takes values in an infinite dimensional Hilbert space. Namely,
let $\mathcal{H}$ be the Hilbert-space of ${H}^1$-curves $\{v_\tau\}_{\tau\in[0,T]}$ in $(T_xM,g_T)$ with $v_0=0$, equipped with the inner product
$ \langle v,w \rangle_{\mathcal{H}}=\int_0^T \left\langle \dot{v}_\tau,\dot{w}_\tau\right\rangle_{(T_xM,g_T)} d\tau$.
Then $\nabla^\cH F:P_{(x,T)}\M\to \mathcal{H}$ 
is the unique almost everywhere defined function such that
\begin{equation}\label{intro_mall_grad}
 D_V F(\gamma)=\langle \nabla^\cH F(\gamma),v\rangle_{\mathcal{H}},
\end{equation}
for a.e. curve $\gamma$, and every $v\in \mathcal{H}$, where $V=\{P_\tau^{-1} v_\tau\}_{\tau\in [0,T]}$.

Having defined them on cylinder functions, the ($\sigma$-)parallel gradient and the Malliavin gradient can be extended to closed unbounded operators on $L^2$, see Section \ref{sec_gradients} for details.

Finally, the \textbf{Ornstein-Uhlenbeck operator} $\mathcal{L}=\nabla^{\cH \ast}\nabla^\cH$ is defined by composing the Malliavin gradient with its adjoint. More generally, there is a family of Ornstein-Uhlenbeck operators $\mathcal{L}_{\tau_1,\tau_2}$ ($0\leq \tau_1<\tau_2\leq T$), which captures the part of the Laplacian coming from the time interval $[\tau_1,\tau_2]$. In particular, $\mathcal{L}=\mathcal{L}_{0,T}$.

\subsubsection{Ricci flow and the gradient estimate}\label{sec_introgradient}

Our first characterization of solutions of the Ricci flow is in terms of an infinite dimensional gradient estimate on the associated path space.
Let $(M,g_t)_{t\in I}$ be smooth family of Riemannian manifolds and let $P_T\M$ be its path space, equipped with the Wiener measure and the parallel gradient.
If $F:P_T\M\to\dR$ is a sufficiently nice function, for instance a cylinder function, one can ask whether one can control the gradient of $\int_{P_T\M}Fd\Gamma_{(x,T)}$
viewed as a function of $x\in M$, in terms of some natural gradient of $F$ viewed as a function on path space.
In fact, the answer to this question turns out to be highly relevant, in that it yields our first characterization of solutions of the Ricci flow.
Namely, we prove that $(M,g_t)_{t\in I}$ evolves by Ricci flow if and only if the gradient estimate
\begin{equation*}
(R2)\qquad \abs{\nabla_x\int_{P_T\M} F d\Gamma_{(x,T)}}\leq \int_{P_T\M}\abs{\nabla^\p F}d\Gamma_{(x,T)} \, ,
\end{equation*}
holds for all function $F\in L^2(P_T\M,\Gamma_T)$ (for a.e. $(x,T)\in\M$).

\begin{remark}\label{rem_gradest} The infinite dimensional gradient estimate (R2) can be thought of as (vast) generalization of the finite dimensional gradient estimate (S2)
for the heat equation. Namely, let $F=u\circ e_\sigma:P_T\M\to M\to \dR$ be a $1$-point cylinder function, and write $s=T-\sigma$. By equation \eqref{intro_pushforwardwiener} the
pushforward measure
\begin{equation}\label{intro_push_heat}
e_{\sigma,\ast} d\Gamma_{(x,T)}=d\nu_{(x,T)}(\cdot,s)
\end{equation}
is given by the heat kernel measure $d \nu_{(x,T)}(y,s)=H(x,T\,|\,y,s) d\vol_{g(s)}(y)$, and thus
\begin{equation}
 \int_{P_T\M} F d\Gamma_{(x,T)}= \int_{M} u\, e_{\sigma,\ast} d\Gamma_{(x,T)} =(P_{sT}u)(x).
\end{equation}
Moreover, using \eqref{intro_pargrad} on sees that $\abs{\nabla^\p F}(\gamma)=\abs{\nabla u}_{g_s}\!(e_\sigma(\gamma))$, which together with \eqref{intro_push_heat} implies that
\begin{equation}
 \int_{P_T\M} \abs{\D^\p F} d\Gamma_{(x,T)}= \int_{M} \abs{\D u}\, e_{\sigma,\ast} d\Gamma_{(x,T)} =(P_{sT}\abs{\D u})(x).
\end{equation}
Thus, in the special case of $1$-point cylinder function the estimate (R2) reduces to the finite dimensional heat equation estimate
\begin{equation*}
(\textrm{S2})\qquad \abs{\nabla P_{sT}u}\leq P_{sT}\abs{\D u}. 
\end{equation*}
Of course, there are many more test functions on path space than just $1$-point cylinder function. This is one of the reasons why our infinite dimensional estimate (R2) is strong enough to characterize solutions of the Ricci flow,
while the finite dimensional heat equation estimate (S2) just characterizes supersolutions.
\end{remark}

\subsubsection{Ricci flow and the regularity of martingales}\label{sec_intromartingales}

Our second characterization of solutions of the Ricci flow is in terms of the regularity of martingales on its path space.
Let $(M,g_t)_{t\in I}$ be a smooth family of Riemannian manifolds, and let $P_T\M$ be its path space.
For every function $F\in L^2(P_T\M,\Gamma_{(x,T)})$, we can consider the \textbf{induced martingale} $\{F^\tau\}_{\tau\in [0,T]}$,
\begin{equation}\label{intro_martingale}
 F^{\tau}(\gamma)=\int_{P_{T-\tau}\M} F(\gamma|_{[0,\tau]}\ast \gamma') d\Gamma_{\gamma_\tau}(\gamma'),
\end{equation}
where the integral is over all Brownian curves $\gamma'$ based at $\gamma_\tau$, and $\ast$ denotes concatenation. The family $\{F^\tau\}_{\tau\in [0,T]}$ indeed has the martingale property
$(F^{\tau'})^\tau=F^\tau$ ($\tau'\geq \tau$) and captures how $F$ depends on the $[0,\tau]$-part of the curves, see Section \ref{sec_condexp}.
The \textbf{quadratic variation} $[F^\bullet]_\tau$ of the martingale $\{F^\tau\}_{\tau\in [0,T]}$ is defined by 
$[F^\bullet]_\tau=\lim_{\norm{\{\tau_j\}}\to 0} \sum_k (F^{\tau_{k}}-F^{\tau_{k-1}})^2$,
where the limit is taken in probability, over all partions $\{\tau_j\}$ of $[0,\tau]$ with mesh going to zero, see Section \ref{sec_condexp}.
It turns out that solutions of the Ricci flow can be characterized in terms of certain bounds for $\frac{d[F^\bullet]_\tau}{d\tau}$.
Namely, we prove that $(M,g_t)_{t\in I}$ evolves by Ricci flow if and only if the estimate
\begin{equation*}
(R3)\qquad \int_{P_T\M}\frac{d[F^\bullet]_\tau}{d\tau} d\Gamma_{(x,T)}\leq 2\int_{P_T\M}\abs{\D^\p_\tau F}^2 d\Gamma_{(x,T)}
\end{equation*}
holds for every $F\in L^2(P_T\M,\Gamma_{(x,T)})$ (for all $(x,T)\in\M$).

\begin{remark}
The estimate (R3) is a (vast) generalization of (S3).
Namely, let $F=u\circ e_\sigma:P_T\M\to M\to \dR$ be a $1$-point cylinder function, and write $s=T-\sigma$.
If $\eps>0$, then by \eqref{intro_push_heat} and \eqref{intro_martingale} we have
\begin{equation}
 F^\eps(\gamma)=\int_M u(y) d\nu_{\gamma_\eps}(y,s)=(P_{s,T-\eps}u)(x_\eps).
\end{equation}
Appying this twice and using the short time asymptotics of the heat kernel, one can compute that
\begin{multline}
 \int_{P_T\M}\frac{d[F^\bullet]_\tau}{d\tau}|_{\tau=0}\, d\Gamma_{(x,T)}=\lim_{\eps\to 0}\frac{1}{\eps}\int_{P_T\M}\left(F^\eps-(F^\eps)^0\right)^2 d\Gamma_{(x,T)}\\
=\lim_{\eps\to 0}\frac{1}{\eps}\int_{M}\left((P_{s,T-\eps}u)(z)-\int_M (P_{s,T-\eps}u)(\hat{z})\, d\nu_{(x,T)}(\hat{z}, T-\eps) \right)^2 d\nu_{(x,T)}(z,T-\eps)
=2\abs{\nabla P_{sT}u}^2(x).\nonumber
\end{multline}
Thus, in the special case of 1-point cylinder functions, (R3) for $\tau=0$ reduces to the estimate\footnote{For $\tau\neq 0$,
one gets the estimate $P_{tT}\abs{\nabla P_{st}u}^2\leq P_{sT}\abs{\D u}^2$, which is easily seen to be equivalent to (S3).} 
\begin{equation*}
(\textrm{S3})\qquad \abs{\nabla P_{sT}u}^2\leq P_{sT}\abs{\D u}^2. 
\end{equation*}
\end{remark}

\subsubsection{Ricci flow and the log-Sobolev inequality}\label{sec_intrologsob}

Our third characterization of solutions of the Ricci flow is in terms of a log-Sobolev inequality on its path space.
Log-Sobolev inequalities have a long history, going back to Gross \cite{Gross}.
In the context of Ricci flow, they appear in Perelman's monotonicity formula \cite{Per1} and also in the inequality (S4) of Hein-Naber \cite{HN_eps}.
We characterize solutions of the Ricci flow via an infinite dimensional generalization of the inequality (S4).
Namely, we prove that $(M,g_t)_{t\in I}$ evolves by Ricci flow if and only if the log-Sobolev inequality
\begin{equation*}
 (R4)\qquad  \int_{P_T\M} \left((F^2)^{\tau_2} \log\, (F^2)^{\tau_2} - (F^2)^{\tau_1} \log\, (F^2)^{\tau_1} \right)\, d\Gamma_{(x,T)}\leq 4 \int_{P_T\M} \langle F,\mathcal{L}_{\tau_1,\tau_2} F\rangle d\Gamma_{(x,T)},
\end{equation*}
holds for every $F$ in the domain of the Ornstein-Uhlenbeck operator $\mathcal{L}_{\tau_1,\tau_2}$ (for all $(x,T)\in\M$ and all $0\leq \tau_1<\tau_2\leq T$). Here, $(F^2)^\tau$ denotes the martingale induced by $F^2$. 

\begin{remark} If $\tau_1=0$ and $\tau_2=T$ the inequality (R4) takes the somewhat simpler form
\begin{equation}
 \int_{P_T\M} F^2 \log\, F^2\, d\Gamma_{(x,T)}\leq 4 \int_{P_T\M} \abs{\nabla^\cH F}^2 d\Gamma_{(x,T)}
\end{equation}
for all $F$ with $\int_{P_T\M}F^2=1$. Specializing further, for a 1-point cylinder function $F=u\circ e_\sigma:P_T\M\to M\to \dR$ ($s=T-\sigma$), using \eqref{intro_mall_grad} one can see that  $\abs{\nabla^{\cH} F}_{\mathcal{H}}^2(\gamma)=(T-s)\abs{\D u}_{g_s}\!(e_\sigma(\gamma))$, c.f. Proposition \ref{prop_malliavin}. Together with \eqref{intro_push_heat} this shows that (R4) then reduces to (S4).
\end{remark}

\subsubsection{Ricci flow and the spectral gap}\label{sec_introgap}
Our final characterization of solutions of the Ricci flow is in terms of the spectral gap of the Ornstein-Uhlenbeck operator on its path space.\footnote{It is of course well known that a log-Sobolev inequality implies a spectral gap. However, the important point we prove is that the spectral gap is in fact strong enough to characterize solutions of the Ricci flow.} We prove that $(M,g_t)_{t\in I}$ evolves by Ricci flow if and only if the Ornstein-Uhlenbeck operator $\mathcal{L}_{\tau_1,\tau_2}$ (for all $(x,T)\in \M$ and all $0\leq \tau_1<\tau_2\leq T$) satisfy the spectral gap estimate
\begin{equation*}
(R5)\qquad  \int_{P_T\M}(F^{\tau_2}-F^{\tau_1})^2 d\Gamma_{(x,T)}\leq 2 \int_{P_T\M}\langle F,\mathcal{L}_{\tau_1,\tau_2} F\rangle d\Gamma_{(x,T)}.
\end{equation*}

\begin{remark}\label{rem_spectgap_red}
In the special case of 1-point cylinder functions, the estimate (R5) again reduces to (S5).
\end{remark}

\subsubsection{Summary of main results}
Our main results are summarized in the following theorem.

\begin{thm}[Characterization of solutions of the Ricci flow]\label{main_thm_intro}
For every smooth family $(M,g_t)_{t\in I}$ of Riemannian manifolds (complete, satisfying \eqref{tech_ass}), the following conditions are equivalent:
\begin{enumerate}[(R1)]
 \item \label{R1} The family $(M,g_t)_{t\in I}$ evolves by Ricci flow,
\begin{equation*}
\partial_t g_t= -2\Rc_{g_t}.
\end{equation*}
\item \label{MT2} For every $F\in L^2(P_T\M,\Gamma_T)$, we have the gradient estimate
\begin{equation*}
\abs{\nabla_x\int_{P_T\M} F d\Gamma_{(x,T)}}\leq \int_{P_T\M}\abs{\nabla^\p F}d\Gamma_{(x,T)}.
\end{equation*}
\item \label{R3} For every $F\in L^2(P_T\M,\Gamma_{(x,T)})$, the induced martingale $\{F^\tau\}_{\tau\in [0,T]}$ satisfies the estimate
\begin{equation*}
 \int_{P_T\M}{\frac{d[F^\bullet]_\tau}{d\tau}} d\Gamma_{(x,T)}\leq 2\int_{P_T\M}\abs{\D^\p_\tau F}^2 d\Gamma_{(x,T)}.
\end{equation*}
\item \label{R4} The Ornstein-Uhlenbeck operator $\mathcal{L}_{\tau_1,\tau_2}$ on based path space $L^2(P_T\M,\Gamma_{(x,T)})$ satisfies the log-Sobolev inequality
\begin{equation*}
 \int_{P_T\M} \left((F^2)^{\tau_2} \log\, (F^2)^{\tau_2} - (F^2)^{\tau_1} \log\, (F^2)^{\tau_1} \right)\, d\Gamma_{(x,T)}\leq 4 \int_{P_T\M} \langle F,\mathcal{L}_{\tau_1,\tau_2} F\rangle \, d\Gamma_{(x,T)}.
\end{equation*}
\item \label{R5} The Ornstein-Uhlenbeck operator $\mathcal{L}_{\tau_1,\tau_2}$ on based path space $L^2(P_T\M,\Gamma_{(x,T)})$ satisfies the spectral gap estimate
\begin{equation*}
 \int_{P_T\M}(F^{\tau_2}-F^{\tau_1})^2 d\Gamma_{(x,T)}\leq 2 \int_{P_T\M}\langle F,\mathcal{L}_{\tau_1,\tau_2} F\rangle d\Gamma_{(x,T)}.
\end{equation*}
\end{enumerate}
\end{thm}

\begin{remark}
As explained above, in the special case of 1-point cylinder functions the estimates (R2)--(R5) reduce to the estimates (S2)--(S5), respectively.
\end{remark}

\begin{remark}
Further characterizations are possible. In particular, we have an $L^2$-version of the gradient estimate, and a pointwise $L^1$-version of the martingale estimate, see (R2') and (R3') in Section \ref{sec_proof_main}.
\end{remark}

\noindent\textbf{Outline.} This article is organized as follows. In Section \ref{sec_supersol}, as a warmup for the proof of the main theorem, we prove Theorem \ref{thm_supersol} characterizing supersolutions of the Ricci flow. In Section \ref{sec_prelim}, we set up the machinery of stochastic analysis in our setting of evolving manifolds. In Section \ref{sec_proof_main}, we prove the main theorem (Theorem \ref{main_thm_intro}) characterizing solutions of the Ricci flow.

\section{Supersolutions of the Ricci flow}\label{sec_supersol}

In this short section we prove Theorem \ref{thm_supersol}, characterizing supersolutions of the Ricci flow

\begin{proof}[Proof of Theorem \ref{thm_supersol}] We will prove the implications (S3)$\Leftrightarrow$(S1)$\Leftrightarrow$(S2) and (S1)$\Rightarrow$(S4)$\Rightarrow$(S5)$\Rightarrow$(S3).

(S1)$\Leftrightarrow$(S3): If $g_t$ is a supersolution of the Ricci flow, then the Bochner formula (\ref{parboch}) implies
\begin{equation}
\Box \abs{\D P_{st}u}^2\leq 0.
\end{equation}
Thus, $\abs{\D P_{st} u}^2- P_{st}\abs{\D u}^2$ is a subsolution of the heat equation. Since it is zero for $t=s$, it stays nonpositive for all $t>s$, in particular $\abs{\D P_{sT} u}^2\leq P_{sT}\abs{\D u}^2$.
To prove the converse implication, assume that $(\partial_tg+2\Rc)(X,X)<0$ for some unit tangent vector $X\in T_xM$ at some time $s$. Choose a test function $u$ with $\D u(p)=X$ and $\D^2 u(p)=0$. Then by (\ref{parboch}) we have $\partial_t\abs{\D P_{st}u}^2>\Lap \abs{\D u}^2$ at $p$ at $t=s$; this contradicts (S3).

(S1)$\Leftrightarrow$(S2): If $g_t$ is a supersolution of the Ricci flow, then using the Bochner formula (\ref{parboch}) and the Cauchy-Schwarz inequality we obtain
\begin{align}
\Box \abs{\D P_{st} u} &= \frac{1}{\abs{\D P_{st} u}}\left(\frac{1}{2}\Box \abs{\D P_{st} u}^2+\frac{1}{4}\frac{\abs{\D \abs{\D P_{st} u}^2}^2}{\abs{\D P_{st} u}^2}\right)\leq 0.
\end{align}
Thus, $\abs{\D P_{st} u}- P_{st}\abs{\D u}$ is a subsolution of the heat equation. Since it is zero for $t=s$, it stays nonpositive for all $t>s$, in particular $\abs{\D P_{sT} u}\leq P_{sT}\abs{\D u}$. The converse implication follows by considering a test function as above.

(S1)$\Rightarrow$(S4): Let $w>0$. We start by deriving another estimate for the heat equation. Using the Bochner formula (\ref{parboch}) and the Peter-Paul inequality we compute
\begin{align}
\Box \left(\frac{\abs{\D P_{st} w}^2}{P_{st}w}\right) = \frac{\Box\abs{\D P_{st}w}^2}{P_{st}w}+2\frac{\ip{\D \abs{\D P_{st}w}^2,\D P_{st}w}}{(P_{st}w)^2}-2\frac{\abs{\D P_{st}w}^4}{(P_{st}w)^3}\leq 0.
\end{align}
Thus, $\frac{\abs{\D P_{st} w}^2}{P_{st}w}-P_{st}\frac{\abs{\D w}^2}{w}$ is a subsolution of the heat equation.
Since it is zero for $t=s$, this implies the estimate
\begin{equation}\label{app_frac}
\frac{\abs{\D P_{sr} w}^2}{P_{sr}w}\leq P_{sr}\frac{\abs{\D w}^2}{w}.
\end{equation}
Now, using the heat kernel homotopy principle \cite[(3.7)]{HN_eps} and \eqref{app_frac} we compute
\begin{equation}
 \int w\log w \, d\nu-\left(\int w\, d\nu\right) \log\left( \int w\, d\nu\right)=\int_s^T \left(P_{rT}\frac{\abs{\D P_{sr} w}^2}{P_{sr}w}\right) (x)\, dr\leq (T-s)\int \frac{\abs{\D w}^2}{w}\, d\nu.
\end{equation}
Substituting $w=u^2$ this implies the log-Sobolev inequality (S4).

(S4)$\Rightarrow$(S5):
This follows by evaluating (S4) for $w^2=1+\varepsilon u$ with $\int u\, d\nu=0$.

(S5)$\Rightarrow$(S3)
By the heat kernel homotopy principle \cite[(3.7)]{HN_eps} we have
\begin{equation}
 \int u^2 \, d\nu-\left(\int u d\nu\right)^2 =2\int_s^T \left( P_{rT}\abs{\D P_{sr}u}^2\right) (x)\, dr.
\end{equation}
Thus, if (S3) fails at some $(x,T)$, then (S5) fails for $d\nu_{(x,T)}$ with $\abs{T-s}$ small enough.
\end{proof}

\section{Stochastic calculus on evolving manifolds}\label{sec_prelim}

We will now discuss in more detail the required background from stochastic analysis, adapted to our time-dependent setting.
There are numerous excellent references for stochastic analysis on manfolds, e.g. \cite{Elworthy,Emery,Hsu,IkedaWatanabe,Malliavin,Stroock}.
For readers who wish to focus on one single reference which is particularly close in spirit to the content of the present section we recommend the book by Hsu \cite{Hsu}.

\subsection{Frame bundle on evolving manifolds}\label{sec_diffgeo}

To set things up efficiently, we will first explain how to formulate the differential geometry of evolving manifolds in terms of the frame bundle.
For the frame bundle formalism in the time-independent case, see e.g. Kobayashi-Nomizu \cite{KobNom}, for the frame bundle formalism for the Ricci flow, see Hamilton \cite{Hamilton_Harnack}.\\

Let $(M,g_t)_{t\in I}$, $I=[0,T_1]$, be a smooth family of Riemannian manifolds, and write $\M=M\times I$. Let $Y$ be a time dependent vector field.
For each $X\in (T_{x}M,g_t)$ we can compute the covariant spatial derivative $\nabla_X Y=\nabla^{g_t}_X Y$ using the Levi-Civita connection of the metric $g_t$.
The covariant time derivative is defined as $\nabla_t Y=\partial_t Y+\frac{1}{2}\partial_t g_t(Y,\cdot)^{\sharp_{g_t}}$.
The point is that this gives metric compatibility, namely $\frac{d}{dt}\abs{Y}_{g_t}^2=2\langle Y,\nabla_t Y\rangle$.

Consider the $O_n$-bundle $\pi: \F\to \M$, where the fibres $\F_{(x,t)}$ are given by the orthogonal maps $u: \dR^n\to (T_xM,g_t)$, and $g\in O_n$ acts from the right via composition. 
The horizontal lift of a curve $\gamma_t$ in $\M$ is a curve $u_t$ in $\F$ with $\pi u_t=\gamma_t$ such that $\nabla_{\dot \gamma_t}(u_te)=0$ for all $e\in\dR^n$. 
Given a vector $\alpha X+\beta \partial_t\in T_{(x,t)}\M$ and a frame $u\in \F_{(x,t)}$ there is a unique horizontal lift $\alpha X^*+\beta D_t$ with $\pi_* (\alpha X^*+\beta D_t)=X$. Here, $X^*$ is just the horizontal lift of $X\in T_xM$ with respect to the fixed metric $g_t$, and $D_t=\frac{d}{ds}|_0 u_s$, where $u_s$ is the horizontal lift based at $u$ of the curve $s\mapsto (x,t+s)$ with $x$ constant.
Most of the time we only consider curves of the form $\gamma_\tau=(x_\tau,T-\tau)$.
We denote space-time parallel transport by $P_{\tau_1,\tau_2}=u_{\tau_2}u_{\tau_1}^{-1}:(T_{x_{\tau_1}}M,g_{T-\tau_1})\to (T_{x_{\tau_2}}M,g_{T-\tau_2})$,
and observe that this induces parallel translation maps for arbitrary tensor fields. We write $D_\tau=-D_t$.

Given a representation $\rho$ of $O_n$ on some vector space $V$ and an equivariant map from $\F$ to $V$, we get a section of the associated vector bundle $\F\times_\rho V$, and vice versa.
For example, a time-dependent function $f$ corresponds to the invariant function $\tilde{f}=f\pi:\F\to\dR$, and a time-dependent vector field $Y$ corresponds to a function $\tilde{Y}:\F\to \dR^n$ via $\tilde{Y}(u)=u^{-1}Y_{\pi u}$, which is equivariant in the sense that $\tilde{Y}(ug)=g^{-1}\tilde{Y}(u)$.
The following lemma shows how to compute derivatives in terms of the frame bundle.

\begin{lemma}[First derivatives]\label{lemma_firstder}
$\widetilde{Xf}=X^*\tilde f$, $\widetilde{\partial_t f}=D_t\tilde{f}$, $\widetilde{\nabla_X Y}=X^*\tilde{Y}$, and $\widetilde{\nabla_t Y}=D_t\tilde{Y}$.
\end{lemma}

\begin{proof}
The first two formulas are obvious, since the horizontal lift of a function is constant in fibre direction. To prove the last formula, let $u_t$ be a horizontal curve with $\pi u_t=\gamma_t=(x,t)$, where $x$ is fixed. Then
\begin{equation}
(D_t\tilde{Y})_{u_{t_1}}=\frac{d}{dt}|_{t_1}\tilde{Y}(u_t)=\frac{d}{dt}|_{t_1} u_t^{-1}Y_{\pi u_t}= u_{t_1}^{-1}\frac{d}{ds}|_0 P_{t_1,t_1+s}^{-1}Y_{(x,t_1+s)}=u_{t_1}^{-1}(\nabla_t Y)_{(x,t_1)}=(\widetilde{\nabla_t Y})_{u_1}.
\end{equation}
The third formula follows from a similar computation. In fact, it is a well known formula from differential geometry with respect to a fixed metric $g_t$.
\end{proof}

Let $e_1,\ldots,e_n$ be the standard basis of $\dR^n$.
We write $H_i$ for the horizontal vector fields $H_i(u)=(ue_i)^*$, where $*$ denotes the horizontal lift, as before. The horizontal Laplacian is defined by  $\Lap_H=\sum_{i=1}^n H_i^2$.
\begin{lemma}[Laplacian]\label{lemma_laplacian}
$\widetilde{\Lap f}=\Lap_H\tilde{f}$, $\widetilde{\Lap Y}=\Lap_H\tilde{Y}$.
\end{lemma}

\begin{proof} This is a classical fact from differential geometry with respect to a fixed metric $g_t$.
\end{proof}

We also need the notion of the antidevelopment of a horizontal curve (this concept is also known as Cartan's rolling without slipping), see e.g. \cite{KobNom}, generalized to the time-dependent setting.
The point is that the horizontal vector fields provide a way to identify curves in $\dR^n$ with horizontal curves in $\mathcal{F}$.

\begin{definition}[Antidevelopment] If $\{u_\tau\}_{\tau\in[0,T]}$ is a horizontal curve in $\mathcal{F}$ with $\pi(u_\tau)=(x_\tau,T-\tau)$, its \textbf{antidevelopment} $\{w_\tau\}_{\tau\in[0,T]}$ is the curve in $\dR^n$ that satisfies
\begin{equation}\label{antidevelop}
\frac{du_\tau}{d\tau} =D_\tau+ H_i(u_\tau)\frac{d{w}_\tau^i}{d\tau},\qquad w_0=0.
\end{equation}
\end{definition}

\subsection{Brownian motion and stochastic parallel transport}\label{sec_brownian}

The goal of this section is to generalize the Eells-Elworthy-Malliavin construction of Brownian motion and stochastic parallel translation,
see e.g. \cite{Hsu}, to our setting of evolving manifolds. We note that a related construction in the time-dependent setting has been given by Arnoudon-Coulibaly-Thalmaier \cite{ACT}.\\

The idea is to solve \eqref{antidevelop} in a stochastic setting.
This provides a way to identify Brownian curves  $\{w_\tau\}_{\tau\in[0,T]}$ in $\dR^n$ with horizontal Brownian curves  $\{u_\tau\}_{\tau\in[0,T]}$ in $\mathcal{F}$.
The virtue of this approach is that it yields both Brownian motion on $M$, via projecting, and stochastic parallel transport, via $P_{\tau_1,\tau_2}=u_{\tau_2}u_{\tau_1}^{-1}$.\\

Let $(M,g_t)_{t\in I}$, $I=[0,T_1]$, be a one-parameter family of Riemannian manifolds, and let $\pi: \F\to M\times I$ be the time dependent $O_n$-bundle introduced in the previous section.
We fix a frame $u\in\F$, write $\pi(u)=(x,T)$, and denote the projections to space and time by $\pi_1:\F\to M$ and $\pi_2:\F\to I$, respectively.
It will be convenient to work with the backwards time $\tau$, defined by $t=T-\tau$. As before, we write $D_\tau=-D_t$.

Motivated by (\ref{antidevelop}), we consider the following stochastic differential equation (SDE) on $\F$:
\begin{equation}\label{SDE}
dU_\tau=D_\tau d\tau+ H_i(U_\tau)\circ {dW}_\tau^i\, ,\qquad\qquad U_0=u.
\end{equation}
Here, $W_\tau$ is Brownian motion on $\dR^n$, and $\circ$ indicates that the equation is in the Stratonovich sense.
To keep the factor $2$ in Hamilton's Ricci flow, $\partial_t g_t=-2\Rc_{g_t}$, we use the convention that $dW_\tau$ doesn't have the standard normalization from stochastic calculus, but is scaled by a factor $\sqrt{2}$, i.e. 
$dW_\tau^i dW_\tau^j=2\delta_{ij}d\tau$.

\begin{proposition}[Existence, uniqueness, and Ito formula]\label{prop_SDE}
The SDE \eqref{SDE} has a unique solution $\{U_\tau\}_{\tau\in [0,T]}$. The solution satisfies $\pi_2(U_\tau)=T-\tau$, and does not explode. Moreover, $U_\tau(\omega)$ is continuous in $\tau$ for almost every Brownian path $\omega\in C([0,T],\dR^n)$, and for any $C^2$-function $f:\F\to \dR$ we have the Ito formula
\begin{equation}\label{ito}
df(U_\tau)=H_if(U_\tau)dW^i_\tau+D_\tau f(U_\tau)d\tau+ H_iH_if(U_\tau)d\tau.
\end{equation}
\end{proposition}

\begin{proof}
We recall that SDEs on manifolds can be reduced to SDEs on Euclidean space, see e.g. \cite[Sec. 1.2]{Hsu}. 
Choose an embedding $\F\subset \dR^N$ and suitable extensions of all functions to $\dR^N$. 
By the standard theory of SDEs on Euclidean space, there is a unique solution of the system ($a=1,\ldots, N$):
\begin{equation}\label{SDE_eucl}
dU^a_\tau=D^a_\tau d\tau+ H^a_i(U_\tau)\circ {dW}_\tau^i\, ,\qquad\qquad U_0=u.
\end{equation}
It follows from a Gronwall type argument that the solution actually stays inside $\F$, see e.g. \cite[Prop. 1.2.8]{Hsu}. This proves existence of a solution of \eqref{SDE}.
Moreover, it is also easy to derive a uniqueness result for solutions of \eqref{SDE} from the standard uniqueness result for SDEs on Euclidean space, see e.g. \cite[Thm. 1.2.9]{Hsu}.
In particular, the solution is independent of the choices of embedding and extensions.
Since Brownian motion in $\mathbb{R}^n$ is continuous in $\tau$ for almost every path, the same is true for $U_\tau$.

To prove \eqref{ito}, we first convert \eqref{SDE_eucl} into a SDE in the Ito sense. Computationally this is done by dropping the $\circ$ and adding one half times the quadratic variation of $H(U_\tau)$ and $W_\tau$:
\begin{equation}\label{SDE_ito}
dU^a_\tau=D^a_\tau d\tau+ H^a_i(U_\tau) {dW}_\tau^i+\tfrac{1}{2}dH_{i}^a(U_\tau)dW_\tau^i\, ,\qquad\qquad U_0=u.
\end{equation}
Now, using Ito calculus in Euclidean space we compute
\begin{equation}
dH_{i}^a(U_\tau)dW_\tau^i=\partial_bH_{i}^a(U_\tau)dU_\tau^bdW_\tau^i=2\partial_bH_{i}^a(U_\tau)H_i^b(U_\tau)d\tau,
\end{equation}
and
\begin{align}
df(U_\tau)=&\partial_af(U_\tau)dU_\tau^a+\tfrac{1}{2}\partial_a\partial_bf(U_\tau)dU_\tau^adU_\tau^b\nonumber\\
=&\partial_af(U_\tau)D^a_\tau d\tau+\partial_af(U_\tau)H^a_i(U_\tau) {dW}_\tau^i\\
&+\left(\partial_af(U_\tau)\partial_bH_{i}^a(U_\tau)H_i^b(U_\tau)+\partial_a\partial_bf(U_\tau)H^a_i(U_\tau)H^b_i(U_\tau)\right)d\tau.\nonumber
\end{align}
Observing that the term in brackets is equal to $H_iH_if(U_\tau)$, this proves \eqref{ito}.

By assumption \eqref{tech_ass} the metrics are equivalent at all times and there exists a distance-like function, i.e. a smooth function $r:M\to \mathbb{R}$ such that, after fixing an arbitrary point and $o\in M$,
\begin{equation}
 C^{-1}(1+d_t(x,o))\leq r(x)\leq C(1+d_t(x,o)),\qquad \abs{\D r} \leq C,\qquad \D\D r\leq C
\end{equation}
for some $C<\infty$. Let $\tilde{r}:\F\to \mathbb{R}$ be the extension of $r$, that is independent of time and the fibre coordinates. Applying the Ito formula \eqref{ito} to $\tilde{r}$, we see that the solution of \eqref{SDE_eucl} does not explode, i.e. that $U_\tau$ does not escape to spatial infinity.
Finally, for $f=\pi_2$ the Ito formula takes the simple form $d\pi_2(U_\tau)=-d\tau$.  Together with $\pi_2(U_0)=T$, this implies that $\pi_2(U_\tau)=T-\tau$.
\end{proof}

Using Propositon \ref{prop_SDE} we can now define Brownian motion and stochastic parallel transport on our evolving family of Riemannian manifolds.

\begin{definition}[Brownian motion]\label{def_brownian_motion} We call $\pi(U_\tau)=(X_\tau,T-\tau)$ \textbf{Brownian motion} based at $(x,T)$.
\end{definition}

\begin{definition}[Stochastic parallel transport]\label{def_stoch_par} The family of isometries
$P_\tau=U_0U_\tau^{-1}: (T_{X_{\tau}}M,g_{T-\tau})\to (T_{x}M,g_T)$, depending on $\tau$ and the Brownian curve, is called \textbf{stochastic parallel transport}.
\end{definition}

Brownian motion comes naturally with its path space, diffusion measure, and filtered $\sigma$-algebra. 

\begin{definition}[Based path spaces]
We let $P_0\dR^n$ be the space of continuous curves $\{\omega_\tau\}_{\tau\in[0,T]}$ in $\dR^n$ with $\omega_0=0$,
let $P_u\F$ be the space of continuous curves $\{u_\tau\}_{\tau\in[0,T]}$ in $\F$ with $u_0=u$ and $\pi_2(u_\tau)=T-\tau$,
and let $P_{(x,T)}\M$ be the space of continuous curves $\{\gamma_\tau=(x_\tau,T-\tau)\}_{\tau\in[0,T]}$ in $\M$ with $\gamma_0=(x,T)$.
\end{definition}

To introduce the diffusion measure, note that Proposition \ref{prop_SDE} defines a map $U:P_0\dR^n\to P_u\F$, $U(\omega)(\tau)=U_\tau(\omega)$.
We also have a natural map $\Pi:P_u\F\to P_{(x,T)}\M$, induced by the projection $\pi:\F\to M\times I$.

\begin{definition}[Diffusion measures]
Let $\Gamma_0$ be the Wiener measure on $P_0\dR^n$, let $\Gamma_u=U_\ast \Gamma_0$ be the probability measure on $P_u\F$ obtained by pushing forward via $U$, 
and let $\Gamma_{(x,T)}=(\Pi\circ U)_\ast \Gamma_0$ be the probability measure on $P_{(x,T)}\M$ obtained by pushing forward via $\Pi\circ U$.
\end{definition}

Finally, recall the Wiener space $P_0\dR^n$ comes naturally equipped with a filtered family of $\sigma$-algebras $\Sigma^\tau=\Sigma^\tau({P_0\dR^n})$, which is generated by the evaluation maps $e_{\tau_1}:P_0\dR^n\to \dR^n$,
$e_{\tau_1}(\omega)=\omega_{\tau_1}$ with $\tau_1\leq \tau$.

\begin{definition}[Filtered $\sigma$-algebras]\label{def_filtered}
We denote by $\Sigma^\tau(P_u\F)$ and $\Sigma^\tau(P_{(x,T)}\M)$ (or simply by $\Sigma^\tau$ if there is no risk of confusion) the pushforward of $\Sigma^\tau({P_0\dR^n})$ under the maps $U$ and $\Pi \circ U$, respectively.
\end{definition}

\subsection{Conditional expectation and martingales}\label{sec_condexp}

If $F:P_u\F\to\dR$ is integrable, we write $E_{u}[F]=\int_{P_u\F} F d\Gamma_u$ for its \textbf{expectation}.
More generally, if $\sigma\in[0,T]$, we write $F^\sigma=E_{u}[F|\Sigma^\sigma]$ for the \textbf{conditional expectation} given the $\sigma$-algebra $\Sigma^\sigma$ (see Definition \ref{def_filtered}).
We recall that the conditional expectation $F^\sigma$ is the unique $\Sigma^\sigma$-measurable function such that 
$\int_\Omega F^\sigma d\Gamma_u =\int_\Omega F d\Gamma_u$
for all $\Sigma^\sigma$-measurable sets $\Omega$.
Similarly, if $F$ is an integrable function on $P_{(x,T)}\M$, we also write $E_{(x,T)}[F]$ and $F^\sigma=E_{(x,T)}[F|\Sigma^\sigma]$ for its expectation and conditional expectation, respectively.

\begin{proposition}[Conditional expectation]\label{prop_condexp}
If $F:P_{(x,T)}\M\to\dR$ is integrable and $\sigma\in[0,T]$, then for a.e. Brownian curve $\{\gamma_\tau\}_{\tau\in [0,T]}$ the conditional expectation $F^\sigma=E_{(x,T)}[F|\Sigma^\sigma]$ is given by the formula
\begin{equation}\label{eq_condexp}
F^\sigma(\gamma)=\int_{P_{T-\sigma}\M}F(\gamma|_{[0,\sigma]}\ast \gamma')\, d\Gamma_{\gamma_\sigma}(\gamma'), 
\end{equation}
where the integral is over all Brownian curves $\{\gamma'_\tau=(x'_\tau,T-\sigma-\tau)\}_{\tau\in [0,T-\sigma]}$ based at $\gamma_\sigma=(x_\sigma,T-\sigma)$ with respect to the measure $\Gamma_{\gamma_\sigma}$,
and $\gamma|_{[0,\sigma]}\ast \gamma'\in P_{(x,T)}\M$ denotes the concatenation of $\gamma|_{[0,\sigma]}$ and $\gamma'$.
\end{proposition}

\begin{proof}
Using Proposition \ref{prop_SDE} we see that the martingale problem for \eqref{SDE} is well posed.
Thus, by the Stroock-Varadhan principle, c.f. \cite[Thm. 10.1.1]{StroockVaradhan},  we have the strong Markov-property
\begin{equation}\label{eq_strong_markov}
E_{u}[f(U_{\sigma+\tau}^u)|\Sigma^\sigma]=E_{U_\sigma^u}[f(U^{U_\sigma^u}_\tau)]
\end{equation}
for all test functions $f:\mathcal{F}\to \dR$ and all stopping times $\sigma\leq T$,
where $\{U_\tau^{u_0}\}_{\tau\in[0,\pi_2(u_0)]}$ denotes the solution of \eqref{SDE} with initial condition $u_0$.
Pushing forward via $\pi:\mathcal{F}\to\M$, and choosing $\sigma$ constant, equation \eqref{eq_strong_markov} implies
\begin{equation}\label{eq_strong_markov2}
E_{(x,T)}\left[f\left(X_{\sigma+\tau}^{(x,T)}\right)\left|\right.\Sigma^\sigma\right]=E_{\left(X^{(x,T)}_{\sigma},T-\sigma\right)}\left[f\left(X^{\left({X^{(x,T)}_{\sigma}},T-\sigma\right)}_\tau\right)\right]
\end{equation}
for all test functions $f:M\to \dR$. Note that equation \eqref{eq_strong_markov2} is exactly equation \eqref{eq_condexp} for the case that $F$ is the $1$-point cylinder function $f\circ u_{\sigma+\tau}$.\footnote{If $F=f\circ u_{\sigma'}$ is a 1-point cylinder function with $\sigma'\leq\sigma$, then \eqref{eq_condexp} holds true trivially.}
Now, if $F$ is a $k$-point cylinder function, then by conditioning at the first evaluation time we can split up the computation of its (conditional) expectation to computing an expectation of a $1$-point cylinder function and of a $(k-1)$-point cylinder function. Arguing by induction, we infer that \eqref{eq_condexp} holds for all cylinder functions.
Since the cylinder functions are dense in the space of all integrable functions, c.f. Definition \ref{def_filtered}, this proves the proposition.
\end{proof}

For any $F\in L^1(P_T\M,\Gamma_{(x,T)})$, the \textbf{induced martingale} $F^\tau=E_{(x,T)}[F|\Sigma^\tau]$ is defined by taking the conditional expectation with respect to the $\sigma$-algebras $\Sigma^\tau$ for every $\tau\in[0,T]$.
It indeed has the martingale property
\begin{equation}
 E_{(x,T)}[F^{\tau'}|\Sigma^\tau]=F^\tau\qquad (\tau'\geq \tau).
\end{equation}

The \textbf{quadratic variation} of the martingale $F^\bullet=\{F^\tau\}_{\tau\in[0,T]}$ (and more generally of any stochastic process where the following limit exists) is defined by
\begin{equation}\label{eq_qaudrvar}
[F^\bullet]_\tau=\lim_{\norm{\{\tau_j\}}\to 0} \sum_k (F^{\tau_{k}}-F^{\tau_{k-1}})^2,
\end{equation}
where the limit is taken in probability, over all partions $\{\tau_j\}$ of $[0,\tau]$ with mesh going to zero.

Assume now that $F\in L^2(P_T\M,\Gamma_{(x,T)})$. Then the convergence in \eqref{eq_qaudrvar} is not just in probability but also in $L^1$. Moreover, we have
the Ito isometry
\begin{equation}\label{ito_isom}
E\left[[F^\bullet]_{\tau'}-[F^\bullet]_\tau\left|\Sigma^\tau\right.\right]=E\left[(F^{\tau'}-F^\tau)^2\left|\Sigma^\tau\right.\right].
\end{equation}
The differential of $[F^\bullet]_\tau$ takes the form $d[F^\bullet]_\tau=Y_\tau\, d\tau$ for some nonnegative $\Sigma^\tau$-adapted stochastic process $Y$, which we denote by $Y_\tau=\tfrac{d[F^\bullet]_\tau}{d\tau}$.
Using Fatou's lemma and equation \eqref{ito_isom} it can be estimated by
\begin{equation}\label{eq_fatou_est}
 \frac{d[F^\bullet]_\tau}{d\tau}\leq\liminf_{\eps\to 0^+}\frac{1}{\eps}E\left[[F^\bullet]_{\tau+\eps}-[F^\bullet]_\tau\left|\Sigma^\tau\right.\right]=\liminf_{\eps\to 0^+}\frac{1}{\eps}E\left[(F^{\tau+\eps}-F^\tau)^2\left|\Sigma^\tau\right.\right],
\end{equation}
for almost every $\tau$ for almost every $\gamma$.

\subsection{Heat equation and Wiener measure}

The goal of this section is to explain the relationship between the Wiener measure and the heat equation on our evolving manifolds. In particular, we will see that the Wiener measure is indeed characterized by equation \eqref{intro_pushforwardwiener}.
We start with the following representation formula for solutions of the heat equation.

\begin{proposition}[Representation formula for solutions of the heat equation]\label{prop_repformula}
If $s\in[0,T]$, and $w$ is a solution of the heat equation, $\partial_t w=\Lap_{g_t}w$, with $w|_{s}=f\in C^\infty_c(M)$, then $w(x,T)=E_{(x,T)}[f(X_{T-s})]$.
\end{proposition}

\begin{proof}
By Definition \ref{def_brownian_motion} we have $w(X_\tau,T-\tau)=\tilde{w}(U_\tau)$, where $\tilde{w}$ denotes the lift of $w$ to the frame bundle, which is constant in fibre directions.
By the Ito formula (Proposition \ref{prop_SDE}) we have
\begin{align}
d\tilde{w}(U_\tau)=H_i\tilde{w}(U_\tau)\, dW_\tau^i+D_\tau\tilde{w}(U_\tau)\, d\tau+\Lap_H\tilde{w}(U_\tau)\,d\tau,
\end{align}
where  $\Lap_H=\sum_{i=1}^n H_i^2$ is the horizontal Laplacian.
Since $w$ solves the heat equation, the sum of the last two terms vanishes (see Lemma \ref{lemma_firstder} and Lemma \ref{lemma_laplacian}), and by integration we obtain
\begin{align}\label{repform_integrand}
\tilde{w}(U_{T-s})-\tilde{w}(U_0)=\int_0^{T-s} H_i \tilde{w}(U_\tau)\, dW_\tau^i.
\end{align}
Note that $\tilde{w}(U_0)=\tilde{w}(u)=w(x,T)$, and that $\tilde{w}(U_{T-s})=w(X_{T-s},s)=f(X_{T-s})=f(\pi_1 U_{T-s})$. Moreover, after taking expectations the term on the right hand side of  \eqref{repform_integrand} disappears by the martingale property, i.e.
since the integrand is $\Sigma^\tau$-adapted (c.f. Definition \ref{def_filtered}), and since Brownian motion has zero expectation. Thus,
\begin{equation}
w(x,T)=E_u[f(\pi_1 U_{T-s})]=E_{(x,T)}[f(X_{T-s})],
\end{equation}
as claimed.
\end{proof}

\begin{proposition}[Characterization of the Wiener measure]\label{prop_wiener}

If $e_{\bf{\sigma}}:P_{(x,T)}\M\to M^{k}$ is the evaluation map at ${\bf{\sigma}}=\{0\leq \sigma_1\leq\ldots\leq \sigma_k\leq T\}$, given by $e_\sigma(\gamma)=(\pi_1\gamma_{\sigma_1},\ldots, \pi_1\gamma_{\sigma_k})$, and if we write $s_i=T-\sigma_i$, then 
\begin{equation}\label{cor_pushwiener2}
e_{\bf{\sigma},\ast}d\Gamma_{(x,T)}(y_1,\ldots,y_k)=H(x,T|y_1,s_1)d\vol_{g_{s_1}}(y_1)\cdots H(y_{k-1},s_{k-1}|y_k,s_k)d\vol_{g_{s_k}}(y_k).
\end{equation}
Moreover, equation \eqref{cor_pushwiener2} uniquely characterizes the Wiener measure on $P_{(x,T)}\M$.
\end{proposition}

\begin{proof}
By Propositon \ref{prop_repformula} we have the equality
\begin{equation}
\int_M H(x,T|y,s)f(y)d\vol_{g(s)}(y)=\int_{P_{(x,T)}\M}f(\pi_1\gamma_\sigma)d\Gamma_{(x,T)}(\gamma)
\end{equation}
for every test function $f$, say smooth with compact support.
Since these functions are dense in the space of all integrable functions on $M$, this proves \eqref{cor_pushwiener2} for $k=1$.

Now, if $f:M^k\to \dR$ and ${\bf{\sigma}}=\{0\leq \sigma_1\leq\ldots\leq \sigma_k\leq T\}$, then using Proposition \ref{prop_condexp} and what we just proved, the
conditional expectation $(e_{\sigma}^\ast f)^{\sigma_{k-1}}=E_{(x,T)} [e_{\sigma}^\ast f | \Sigma^{\sigma_{k-1}}]$ is given by
\begin{equation}
 (e_{\sigma}^\ast f)^{\sigma_{k-1}}(\gamma)=\int_M f(\pi_1\gamma_{\sigma_1},\ldots,\pi_1\gamma_{\sigma_{k-1}},y_k)H(\pi_1\gamma_{\sigma_{k-1}},s_{k-1}|y_k,s_k)d\vol_{g_{s_k}}(y_k).
\end{equation}
Using the formula $E_{(x,T)}[e_{\sigma}^\ast f]=E_{(x,T)}[E_{(x,T)}[e_{\sigma}^\ast f | \Sigma^{\sigma_{k-1}}]]$ and induction, this proves \eqref{cor_pushwiener2}.

Finally, by the density of cylinder functions in the space of measurable functions (c.f. Definition \ref{def_filtered}), equation \eqref{cor_pushwiener2} uniquely characterizes the Wiener measure on $P_{(x,T)}\M$.
\end{proof}

\subsection{Feynman-Kac formula}

We will now prove a Feynman-Kac type formula for vector valued solutions of the heat equation with potential
\begin{equation}\label{eq_vvhwp}
\D_t Y=\Lap_{g_t} Y+A_t Y,\qquad Y|_{s}=Z,
\end{equation}
where $A_t\in \textrm{End}(TM)$ is a smooth family of endomorphisms, and $Z$ is say smooth with compact support.

The idea is to generalizes the representation formula for solutions of the heat equation (Proposition \ref{prop_repformula}) in two ways by: i) using stochastic parallel translation (Definition \ref{def_stoch_par}) to transport everything to $T_xM$, and ii) multiplication by an endomorphism $R_{T-s}=R_{T-s}(\gamma):T_xM\to T_xM$, which is obtained by solving an ODE along every Brownian curve $\gamma$, to capture how the potential $A_t$ effects the solution.

\begin{proposition}[Feynman-Kac formula]\label{prop_feynman_kac}
If $s\in[0,T]$, $A_t\in \textrm{End}(TM)$, and $Y$ is a vector valued solution of the heat equation with potential, $\D_t Y=\Lap_{g_t} Y+A_t Y$, with $Y|_{s}=Z\in C^\infty_c(TM)$, then
\begin{equation}\label{eq_feynman_kac}
Y(x,T)=E_{(x,T)}[R_{T-s} P_{T-s}Z(X_{T-s})],
\end{equation}
where $R_\tau=R_\tau(\gamma):T_xM\to T_xM$ is the solution of the ODE $\tfrac{d}{d\tau} R_{\tau}=R_\tau P_\tau A_{T-\tau} P_\tau^{-1}$ with $R_0=\textrm{id}$.
\end{proposition}

\begin{remark}
Similar formulas hold for tensor valued solutions of the heat equation with potential.
\end{remark}

\begin{proof}
Let $\tilde{Y}:\F\to \mathbb{R}^n$, $\tilde{Y}(u)=u^{-1}Y_{\pi u}$, be the equivariant function associated to $Y$.
Applying the Ito formula  (Proposition \ref{prop_SDE}) to each component, we obtain
\begin{equation}
d\tilde{Y}(U_\tau)=H_i\tilde{Y}(U_\tau)dW^i_\tau+D_\tau\tilde{Y}(U_\tau)d\tau+\Lap_H\tilde{Y}(U_\tau)d\tau=H_i\tilde{Y}(U_\tau)dW^i_\tau-\tilde{A}_{T-\tau}\tilde{Y}(U_\tau)d\tau,
\end{equation}
where we lifted equation \eqref{eq_vvhwp} to $\F$ using Lemma \ref{lemma_firstder} and Lemma \ref{lemma_laplacian}.
Let $\tilde{R}_\tau:\RR^n\to \RR^n$ be the solution of the ODE $\tfrac{d}{d\tau} \tilde{R}_{\tau}=\tilde{R}_\tau \tilde{A}_{T-\tau}$ with $R_0=\textrm{id}$. Then
\begin{equation}\label{eq_proof_of_feynkac}
 d\left(\tilde{R}_\tau \tilde{Y}(U_\tau)\right)=\tilde{R}_\tau H_i\tilde{Y}(U_\tau)dW^i_\tau.
\end{equation}
The right hand side disappears after taking expectations, by the martingale property, as in the proof of Proposition \ref{prop_repformula}. Thus,
\begin{equation}
 \tilde{Y}(u)=E_u[\tilde{R}_{T-s}\tilde{Y}_{T-s}(U_{T-s})].
\end{equation}
Finally, we can translate from $\tilde{Y}$ to $Y$ by computing
\begin{equation}
 Y(x,T)=u\tilde{Y}(u)=E_u[U_0\tilde{R}_{T-s}U_0^{-1}U_0U^{-1}_{T-s}U_{T-s}\tilde{Y}_{T-s}(U_{T-s})]=E_{(x,T)}[R_{T-s}P_{T-s}Z(X_{T-s})].
\end{equation}
Here, we used that $\R_\tau= U_0 \tilde{R}_\tau U_0^{-1}$, which can be checked by computing
\begin{equation}
\tfrac{d}{d\tau} (U_0 \tilde{R}_\tau U_0^{-1})=U_0\tilde{R}_\tau \tilde{A}_{T-\tau}U_0^{-1}
=U_0\tilde{R}_\tau U_0^{-1} U_0 U_\tau^{-1} U_\tau \tilde{A}_{T-\tau}U_\tau^{-1}U_\tau U_0^{-1}
=U_0\tilde{R}_\tau U_0^{-1} P_\tau A_{T-\tau} P_\tau^{-1},\nonumber
\end{equation}
which shows that $\R_\tau$ and $U_0\tilde{R}_\tau U_0^{-1}$ solve the same ODE, and thus must be equal.
\end{proof}

\subsection{Parallel gradient and Malliavin gradient}\label{sec_gradients}

Let $F:P_{(x,T)}\M\to \dR$ be a cylinder function. If $\gamma\in P_{(x,T)}\M$ is a continuous curve and $V$ is a right continuous vector field along $\gamma$, then the 
directional derivative $D_VF(\gamma)$ is well defined as a limit of difference quotients, namely
\begin{equation}
 D_V F(\gamma)=\lim_{\varepsilon\to 0} \frac{F(\gamma^{V,\eps})-F(\gamma)}{\varepsilon},
\end{equation}
where $\gamma^{V,\eps}=\{(x^{V,\eps}_\tau,T-\tau)\}_{\tau\in [0,T]}$ is the curve in $P_{(x,T)}\M$ defined by $x^{V,\eps}_\tau=\exp^{g_\tau}_{x_\tau}(\eps V_\tau)$.

\begin{definition}[Parallel gradient]\label{def_par_grad}
Let $\sigma\in[0,T]$. If $F:P_{(x,T)}\M\to \dR$ is a cylinder function, then its \textbf{$\sigma$-parallel gradient}
is the unique almost everywhere defined function $\nabla^\p_\sigma F:P_{(x,T)}\M \to (T_{x}M,g_T)$, such that
\begin{equation}
 D_{V^\sigma} F(\gamma)=\langle \nabla^\p_\sigma F(\gamma),v\rangle_{(T_xM,g_T)}
\end{equation}
for almost every Brownian curve $\gamma$ and every $v\in (T_{x}M,g_T)$, where $V^\sigma=\{V^\sigma_\tau\}_{\tau\in [0,T]}$ is the vector field along $\gamma$ given by $V^\sigma_\tau=0$ if $\tau\in[0,\sigma)$ and $V^\sigma_\tau=P_\tau^{-1} v$ if $\tau\in[\sigma,T]$.
\end{definition}

Explicitly, if $F=u\circ e_\sigma: P_{(x,T)}\M\to M^k\to \dR$, and if we write $s_j=T-\sigma_j$, then it is straightforward to check that
\begin{equation}\label{eqn_par_grad}
 \nabla^\p_\sigma F=e_\sigma^\ast\left(\sum_{\sigma_j\geq \sigma} P_{\sigma_j}\grad_{g_{s_j}}^{(j)}u\right),
\end{equation}
where $\grad^{(j)}$ denotes the gradient with respect to the $j$-th variable, and $P_{\sigma_j}$ is stochastic parallel transport.\\

Let $\mathcal{H}$ be the Hilbert-space of $H^1$-curves $\{v_\tau\}_{\tau\in[0,T]}$ in $(T_xM,g_T)$ with $v_{0}=0$, equipped with the inner product
\begin{equation}
 \langle v,w \rangle_{\mathcal{H}}=\int_{0}^{T} \langle \dot{v}_\tau,\dot{w}_\tau \rangle_{(T_xM,g_T)}\, d\tau.
\end{equation}

\begin{definition}[Malliavin gradient]\label{def_mall_grad}
If $F:P_{(x,T)}\M \to \dR$ is a cylinder function, then its \textbf{Malliavin gradient}
is the unique almost everywhere defined function $\nabla^\cH F:P_{(x,T)}\M\to \mathcal{H}$, such that
\begin{equation}
 D_V F(\gamma)=\langle \nabla^\cH F(\gamma),v\rangle_{\mathcal{H}}
\end{equation}
for every $v\in \mathcal{H}$ for almost every Brownian curve $\gamma$, where $V=\{P_\tau^{-1} v_\tau\}_{\tau\in [0,T]}$. 
\end{definition}

Let us now explain the extension to operators on $L^2$. This is based on the integration by parts formula from the appendix (Theorem \ref{thm_intbyparts}), which says that the formal adjoint of $D_V$ is given by
\begin{equation}\label{eq_adj1}
D_V^\ast G=-D_V G + \tfrac12 G \int_0^T\langle  \tfrac{d}{d\tau} v_\tau-P_{\tau}(\Ric+\tfrac{1}{2}\partial_t{g})P_\tau^{-1} v_\tau,dW_\tau\rangle.
\end{equation}
By the Ito isometry and \eqref{tech_ass} we have the estimate
\begin{equation}\label{eq_adj2}
E_{(x,T)}\left|\int_0^T\langle  \tfrac{d}{d\tau} v_\tau-P_{\tau}(\Ric+\tfrac{1}{2}\partial_t{g})P_\tau^{-1} v_\tau,dW_\tau\rangle\right|^2\leq C\abs{v}_{\mathcal{H}}^2.
\end{equation}
Using \eqref{eq_adj1}, \eqref{eq_adj2}, and the definition of the formal adjoint, we see that if $F_n$ is a sequence of cylinder functions with $F_n\to 0$ and $D_V F_n\to K$ in $L^2(P_{(x,T)}\M)$, then $(K,G)=0$ for all cylinder functions $G$, and thus $K=0$.
It follows that $\nabla^\cH$ can be extended to a closed unbounded operator from $L^2(P_{(x,T)}\M)$ to $L^2(P_{(x,T)}\M,\mathcal{H})$, with the cylinder functions being a dense subset of the domain.
Similarly, $\nabla^\p_\sigma$ can be extended to a closed unbounded operator from $L^2(P_{(x,T)}\M)$ to $L^2(P_{(x,T)}\M,T_xM)$, again with the cylinder functions being a dense subset of the domain.

\subsection{Ornstein-Uhlenbeck operator}
The \textbf{Ornstein-Uhlenbeck operator} $\mathcal{L}=\nabla^{\cH \ast}\nabla^\cH$ is an unbounded operator on $L^2(P_T\M,\Gamma_{(x,T)})$ defined by composing the Malliavin gradient with its adjoint. More generally, there is a family of 
Ornstein-Uhlenbeck operators $\mathcal{L}_{\tau_1,\tau_2}$ on $L^2(P_T\M,\Gamma_{(x,T)})$ defined by the formula $\mathcal{L}_{\tau_1,\tau_2}=\int_{\tau_1}^{\tau_2} \nabla^{\p \ast}_\tau\nabla^\p_\tau\, d\tau$, which captures the part of the Laplacian coming form the time range $[\tau_1,\tau_2]$. The next proposition shows in particular that $\mathcal{L}=\mathcal{L}_{0,T}$.

\begin{proposition}\label{prop_malliavin}
If $F:P_T\M\to \dR$ is a cylinder function, then for almost every curve $\gamma\in (P_T\M,\Gamma_{(x,T)})$ we have the formula
\begin{equation}
 \abs{\D^\cH F}^2(\gamma)=\int_{0}^{T}\abs{\D^\p_\tau F}^2(\gamma)\, d\tau.
\end{equation}
\end{proposition}

\begin{proof}
 The cylinder function has the form $F=u\circ e_\sigma:P_{T}\M\to M^k\to\dR$. By the definition of the Malliavin gradient (Definition \ref{def_mall_grad}), for almost every $\gamma\in (P_T\M,\Gamma_{(x,T)})$ we have
\begin{equation}
 \sum_{j=1}^k \langle v_{\sigma_j},P_{\sigma_j}\grad^{(j)}_{g_{s_j}}u(e_{\sigma_j}\gamma)\rangle
=D_{V}F(\gamma)
=\langle \D^\cH F(\gamma),v\rangle_\cH
=\int_0^T \langle \tfrac{d}{d\tau} \D^\cH F(\gamma),\tfrac{d}{d\tau} v\rangle\, d\tau.
\end{equation}
It follows that
\begin{equation}
 \tfrac{d}{d\tau} \D^\cH F(\gamma)=\sum_{j=1}^k 1_{\{\tau\leq \sigma_j\}}P_{\sigma_j}\grad^{(j)}_{g_{s_j}}u(e_{\sigma_j}\gamma).
\end{equation}
Based on this, writing $\sigma_0=0$, we compute 
\begin{equation}
 \abs{\D^\cH F}^2_{\mathcal{H}}(\gamma)=\int_0^T \abs{\tfrac{d}{d\tau} \D^\cH F(\gamma)}^2\,d\tau
=\sum_{j=1}^k(\sigma_j-\sigma_{j-1})\,\Big|\!\sum_{\ell=j}^k P_{\sigma_\ell}\grad^{(\ell)}_{g_{s_\ell}}u(e_{\sigma_\ell}\gamma)\Big|^2
=\int_0^T\abs{\D^\p_\tau F}^2(\gamma)\, d\tau,
\end{equation}
where we used that the integrands are piecewise constant. This proves the proposition.
\end{proof}

\section{Proof of the main theorem}\label{sec_proof_main}

In this section, we prove our main theorem (Theorem \ref{main_thm_intro}) characterizing solutions of the Ricci flow.\\

We will prove the implications (R1)$\Rightarrow$(R2)$\Rightarrow$(R3')$\Rightarrow$(R4)$\Rightarrow$(R5)$\Rightarrow$(R3)$\Rightarrow$(R2')$\Rightarrow$(R1).
Here, (R3') denotes the (seemingly stronger) statement that for every $F\in L^2(P_T\M,\Gamma_{(x,T)})$ we have the pointwise estimate
\begin{equation*}
(R3')\qquad \sqrt{{\frac{d[F^\bullet]_\tau}{d\tau}}}(\gamma)\leq \sqrt{2}\, E_{(x,T)}\left[ |{\D^\p_\tau F}|\left| \Sigma^\tau\right.\right](\gamma)
\end{equation*}
for almost every $\gamma\in P_{(x,T)}\M$ for almost every $\tau\in[0,T]$, and (R2') denotes the (seemingly weaker) statement that for every $F\in L^2(P_T\M,\Gamma_T)$, we have the gradient estimate
\begin{equation*}
(R2')\qquad \abs{\nabla_x\int_{P_T\M} F d\Gamma_{(x,T)}}^2\leq \int_{P_T\M}\abs{\nabla^\p F}^2 d\Gamma_{(x,T)}.
\end{equation*}

Before delving into the proof, we observe that it suffices to prove the estimates for cylinder functions, since this implies the general case by approximation. For illustration, let us spell out the approximation argument for (R2):
Let $F\in L^2(P_T\M,\Gamma_T)$. Let $F_j$ be a sequence of cylinder functions that converges to $F$ in $L^2(P_T\M,\Gamma_T)$ and pointwise almost everywhere.
By Fubini's theorem and the dominated convergence theorem, for a.e. $x\in M$ we obtain that $\lim_{j\to\infty} E_{(x,T)}F_j^2=E_{(x,T)}F^2<\infty$. We can assume that for a.e. $x\in M$ the function $F$ is in the domain of the parallel gradient based at $(x,T)$ (since otherwise the right hand side of (R2) is infinite by convention and the estimate holds trivially). Thus, $\lim_{j\to\infty} E_{(x,T)}\abs{\nabla^\p F_j}=E_{(x,T)}\abs{\nabla^\p F}<\infty$ for a.e. $x\in M$. If we know that (R3) holds for cylinder functions, then we can infer that
\begin{equation}\label{eq_loclip}
\limsup_{j\to\infty}\,\big|\nabla_x \!\int_{P_T\M} \!\!\!\!F_j\, d\Gamma_{(x,T)}\big|\leq \int_{P_T\M} \abs{\nabla^\p F}\, d\Gamma_{(x,T)}
\end{equation}
for a.e. $x\in M$. Once we know that the local Lipschitz-bounds \eqref{eq_loclip} holds, then passing to a subsequential limit we can conclude that (R2) holds for $F$ for a.e. $x\in M$.



\subsection{The gradient estimate}\label{subsec_gradest}

The goal of this section is to prove the implication (R1)$\Rightarrow$(R2). We start with the following theorem for the gradient of the expectation value.

\begin{theorem}[Gradient formula]\label{thm_gradient_formula}
If $(M,g_t)_{t\in I}$ is an evolving family of Riemannian manifolds and  $F:P_T\M\to \dR$ is a cylinder function, then
\begin{equation}\label{eq_bismut_form}
\grad_{g_T} E_{(x,T)}F=E_{(x,T)}\left[\D^\p F+\int_0^T \tfrac{d}{d\tau}R_\tau\, \D^\p_\tau F\,d\tau\right],
\end{equation}
where $R_\tau=R_\tau(\gamma):T_xM\to T_xM$ is the solution of the ODE $\tfrac{d}{d\tau}R_\tau=-R_\tau P_\tau(\Ric+\tfrac12 \partial_t g)P_\tau^{-1}$ with $R_0=\textrm{id}$.
\end{theorem}

Our proof of Theorem \ref{thm_gradient_formula} is by induction on the order of the cylinder function.
The main ingredients are the Feynman-Kac formula for vector valued solutions of the heat equation (Proposition \ref{prop_feynman_kac}),
the formula for the conditional expectation value (Proposition \ref{prop_condexp}), and the following evolution equation for the gradient.

\begin{proposition}[Evolution of the gradient]\label{prop_evol_grad}
 If $(M,g_t)_{t\in I}$ is an evolving family of Riemannian manifolds, and $u$ solves the heat equation, $\partial_t u = \Lap_{g_t} u$, then its gradient, $\grad_{g_t} u$, solves the equation
\begin{equation}
 \nabla_t\, \grad_{g_t} u=\Lap_{g_t}\grad_{g_t} u-(\Ric+\tfrac12 \partial_tg_t)(\grad_{g_t} u,\cdot)^{\sharp_{g_t}}.
\end{equation}
\end{proposition}

\begin{proof}
 Using the formula $\partial_t (g^{-1})=-g^{-1}(\partial_tg)g^{-1}$ and the definitions of $\grad_{g_t} (u)$ and  $\D_t$, we compute
\begin{align}
 \nabla_t\, \grad_{g_t} u&=\grad_{g_t} (\partial_t u)-\partial_tg_t(\grad_{g_t} u, \cdot )^{\sharp_{g_t}}+\tfrac{1}{2}\partial_tg_t(\grad_{g_t} u, \cdot )^{\sharp_{g_t}}\nonumber\\
&=\Lap_{g_t}\grad_{g_t} u-(\Ric+\tfrac12 \partial_tg_t)(\grad_{g_t} u,\cdot)^{\sharp_{g_t}},
\end{align}
where we used the equation $\partial_t u = \Lap_{g_t} u$ and commuted the Laplacian and the gradient.
\end{proof}

\begin{proof}[Proof of Theorem \ref{thm_gradient_formula}]
We argue by induction on the order $k=\abs{\sigma}$ of the cylinder function $F=e_\sigma^\ast u$.

If $k=1$, then by equation \eqref{intro_push_heat} the expectation $E_{(x,T)}F$ is given by integration with respect to the heat kernel,
namely
\begin{equation}
 E_{(x,T)}F=\int_M u(y)H(x,T|y,s)d\vol_{g(s)}(y)=(P_{sT}u)(x), 
\end{equation}
where $s=T-\sigma$. On the other hand, by Proposition \ref{prop_evol_grad} we have the evolution equation
\begin{equation}
  \nabla_t\, \grad_{g_t} P_{st}u=\Lap_{g_t}\grad_{g_t} P_{st}u-(\Ric+\tfrac12 \partial_tg_t)(\grad_{g_t} P_{st}u),
\end{equation}
where we view $(\Ric+\tfrac12 \partial_tg_t)$ as endomorphism (using the metric $g_t$).
We can thus apply the Feynman-Kac formula (Proposition \ref{prop_feynman_kac}), and obtain
\begin{equation}
 (\grad_{g_T}P_{sT}u)(x)=E_{(x,T)}[R_\sigma P_{\sigma}(\grad_{g_s} u)(X_{\sigma})],
\end{equation}
where $R_\tau=R_\tau(\gamma):T_xM\to T_xM$ is the solution of the ODE $\tfrac{d}{d\tau}R_\tau=-R_\tau P_\tau(\Ric+\tfrac12 \partial_t g)P_\tau^{-1}$ with $R_0=\textrm{id}$.
Using the fundamental theorem of calculus and equation \eqref{eqn_par_grad}, we can rewrite this as
\begin{align}
 (\grad_{g_T}P_{sT}u)(x)=E_{(x,T)}\left[\left(\textrm{id}+\int_0^\sigma \tfrac{d}{d\tau} R_\tau\, d\tau\right) P_{\sigma}(\grad_{g_s} u)(X_{\sigma})\right]=E_{(x,T)}\left[\D^\p F+\int_0^T \tfrac{d}{d\tau}R_\tau\, \D^\p_\tau F\,d\tau\right].
\end{align}
Thus, the gradient formula \eqref{eq_bismut_form} holds true for $1$-point cylinder functions.

Now, arguing by induction, let $F=e_\sigma^\ast u$ be a $k$-point cylinder function and let $s_i=T-\sigma_i$. Note that 
\begin{equation}\label{eq_doubleexp}
 E_{(x,T)}F=E_{(x,T)}E_{(x,T)}[F|\Sigma^{\sigma_1}],
\end{equation}
Using Proposition \ref{prop_condexp} we see that $G:=E_{(x,T)}[F|\Sigma^{\sigma_1}]$ is a $1$-point cylinder function given by $G=e_{\sigma_1}^\ast w$,
\begin{equation}
 w(y)=E_{(y,s_1)}[u(y,X'_{\sigma_2-\sigma_1},\ldots,X'_{\sigma_k-\sigma_1})],
\end{equation}
where the expectation is over all Brownian curves starting at $(y,T-\sigma_1)$. Note that by equation \eqref{eq_doubleexp} and the case $k=1$ of the gradient formula we have
\begin{equation}
 \grad_{g_T} E_{(x,T)}F=\grad_{g_T} E_{(x,T)}G=E_{(x,T)}R_{\sigma_1} P_{\sigma_1} (\grad_{g_{s_1}} w)(X_{\sigma_1}),
\end{equation}
where $R_\tau=R_\tau(\gamma):T_xM\to T_xM$ is the solution of the ODE $\tfrac{d}{d\tau}R_\tau=-R_\tau P_\tau(\Ric+\tfrac12 \partial_t g)P_\tau^{-1}$ with $R_0=\textrm{id}$.
Using the product rule and induction, we compute
\begin{align}
 (\grad_{g_{s_1}} w)(y)=&E_{(y,s_1)}\grad_{g_s^1}^{(1)}u(y,X'_{\sigma_2-\sigma_1},\ldots,X'_{\sigma_k-\sigma_1})\nonumber\\
&+E_{(y,s_1)}\left[\D'^\p u(y,X'_{\sigma_2-\sigma_1},\ldots,X'_{\sigma_k-\sigma_1})+\int_0^{T-\sigma_1}\!\!\!\!\!\!\!\!\! \tfrac{d}{d\tau}R'_{\tau}\,\D'^\p_\tau u(y,X'_{\sigma_2-\sigma_1},\ldots,X'_{\sigma_k-\sigma_1})\,d\tau\right]
\end{align}
where $X'$ and $\D'^\p$ denotes Brownian motion and the parallel gradient based at $(y,T-\sigma_1)$, and $R'_\tau=R'_\tau(\gamma'):T_yM\to T_yM$ is the solution of the 
ODE $\tfrac{d}{d\tau}R'_\tau=-R'_\tau P'_\tau(\Ric+\tfrac12 \partial_t g)P_\tau'^{-1}$ with $R'_0=\textrm{id}$. Note that
\begin{multline}
E_{(y,s_1)}\grad_{g_s^1}^{(1)}u(y,X'_{\sigma_2-\sigma_1},\ldots,X'_{\sigma_k-\sigma_1})
+E_{(y,s_1)}\D'^\p u(y,X'_{\sigma_2-\sigma_1},\ldots,X'_{\sigma_k-\sigma_1})\\
=\sum_{j=1}^k E_{(y,s_1)} P'_{\sigma_j-\sigma_1}(\grad_{g_{s_j}}^j u)(X'_{\sigma_1-\sigma_1},\ldots,X'_{\sigma_k-\sigma_1}).
\end{multline}
Moreover, if $\gamma=\gamma|_{[0,\sigma_1]}\ast \gamma'$ then $P_{\tau}(\gamma|_{[0,\sigma_1]}\ast \gamma')=P_{\sigma_1}(\gamma)\circ P'_{\tau-\sigma_1}(\gamma')$ and thus
\begin{equation}
 P_{\sigma_1}R'_{\tau-\sigma_1}P_{\sigma_1}^{-1}=R^{-1}_{\sigma_1}R_\tau
\end{equation}
for $\tau\geq \sigma_1$, since both sides solve the same ODE with the same initial condition at time $\sigma_1$.
Putting everything together, we conclude that
\begin{align}
 \grad_{g_T} E_{(x,T)}F=E_{(x,T)}\left[R_{\sigma_1}\D^\p F+\int_{\sigma_1}^T\tfrac{d}{d\tau}R_\tau\, \nabla_\tau^\p F\, d\tau\right]
=E_{(x,T)}\left[\D^\p F+\int_{0}^T\tfrac{d}{d\tau}R_\tau\, \nabla_\tau^\p F\, d\tau\right],
\end{align}
where we also used Proposition \ref{prop_condexp}, the formula $P_{\sigma_j}(\gamma|_{[0,\sigma_1]}\ast \gamma')=P_{\sigma_1}(\gamma)\circ P'_{\sigma_j-\sigma_1}(\gamma')$,
and \eqref{eqn_par_grad}.
\end{proof}

\begin{proof}[Proof of (R1)$\Rightarrow$(R2)]
The gradient formula (Theorem \ref{thm_gradient_formula}), together with the above approximation argument, immediately establishes the implication (R1) $\Rightarrow$ (R2). To see this, just observe that for families of Riemannian manifolds evolving by Ricci flow the time integral in \eqref{eq_bismut_form} vanishes, that
 $\abs{\nabla_x \int_{P_T\mathcal{M}} F\,d\Gamma_x}$ and $\abs{\grad_{g_T} E_{(x,T)}F}$ are the same (just in different notation), and that
$\abs{E_{(x,T)}\D^\p F}\leq \int_{P_T\mathcal{M}}\abs{\D^\p F}\,d\Gamma_{(x,T)}$.
\end{proof}

\subsection{Regularity of martingales}

The goal of this section is to establish the implication (R2)$\Rightarrow$(R3'). For convenience of the reader, we also prove the (obvious and logically not needed) implication (R3')$\Rightarrow$(R3).
We start with the following formula for the quadratic variation of a martingale on path space.

\begin{theorem}[Quadratic variation formula]\label{thm_quadrvar}
 If $(M,g_t)_{t\in I}$ is an evolving family of Riemannian manifolds and $F:P_{(x,T)}\M\to \dR$ is a cylinder function, then
\begin{equation}
 \frac{d[F^\bullet]_\tau}{d\tau}(\gamma)=2\abs{\D_y E_{(y,T-\tau)}F_{\gamma[0,\tau]}}^2(\pi_1\gamma_\tau)
\end{equation}
for almost every $\gamma\in P_{(x,T)}\M$, where $F_{\gamma[0,\tau]}:P_{T-\tau}\M\to\dR$ is defined by $F_{\gamma[0,\tau]}(\gamma')=F(\gamma|_{[0,\tau]}\ast \gamma')$.
\end{theorem}

\begin{proof}[Proof of Theorem \ref{thm_quadrvar}]
Given a cylinder function $F=u\circ e_\sigma:P_{(x,T)}\M\to M^k\to \dR$, and a number $\tau\in[0,T]$, let $j$ be the largest integer such that $\sigma_j\leq \tau$.
By the formula for the conditional expectation (Proposition \ref{prop_condexp}) and the characterization of the Wiener measure (Propositon \ref{prop_wiener}), for $\eps>0$ small enough, $F^{\tau+\eps}$ is given by
\begin{equation}
F^{\tau+\eps}(\gamma)=\int_{M^{k-j}}u(\pi_1\gamma_{\sigma_1},\ldots,\pi_1\gamma_{\sigma_j},y_{j+1},\ldots,y_k) d\nu_{\gamma_{\tau+\eps}}(y_{j+1},s_{j+1})\ldots d\nu_{(y_{k-1},s_{k-1})}(y_k,s_k).
\end{equation}
We can write this as $F^{\tau+\eps}=e_{\tau+\eps}^\ast w_\eps$, where we define $w_\eps=w_{\eps,\gamma_{\sigma_1},\ldots,\gamma_{\sigma_j}}$ by
\begin{equation}
w_\eps(z)=\int_{M^{k-j}}u(\pi_1\gamma_{\sigma_1},\ldots,\pi_1\gamma_{\sigma_j},y_{j+1},\ldots,y_k) d\nu_{(z,T-\tau-\eps)}(y_{j+1},s_{j+1})\ldots d\nu_{(y_{k-1},s_{k-1})}(y_k,s_k).
\end{equation}
Now, since the function $\frac{d[F^\bullet]_\tau}{d\tau}$ is $\Sigma^\tau$-measurable, we can compute
\begin{equation}
\frac{d[F^\bullet]_\tau}{d\tau}(\gamma)=E_{(x,T)}\left[\frac{d[F^\bullet]_\tau}{d\tau}\left|\,\Sigma^\tau\right.\right]=\lim_{\eps\to 0^+}\frac{1}{\eps}E_{(x,T)}\left[(F^{\tau+\eps}-(F^{\tau+\eps})^\tau)^2\left|\,\Sigma^\tau\right.\right],
\end{equation}
where we also used the martingale property $(F^{\tau+\varepsilon})^{\tau}=F^\tau$ and the definition of the quadratic variation, c.f. Section \ref{sec_condexp}.
Using again Proposition \ref{prop_condexp} and Propositon \ref{prop_wiener}, as well as some rough short time asymptotics for the heat kernel, we conclude that
\begin{equation}
\frac{d[F^\bullet]_\tau}{d\tau}(\gamma) =\lim_{\eps\to 0^+}\frac{1}{\eps}\int_{M}\left(w_\eps(z)-\int_M w_\eps(\hat{z})\, d\nu_{\gamma_\tau}(\hat{z}, T-\tau-\eps) \right)^2 d\nu_{\gamma_\tau}(z,T-\tau-\eps)
=2\abs{\nabla w_0}^2\!(\pi_1\gamma_\tau).
\end{equation}
Observing that $w_0(y)=E_{(y,T-\tau)}F_{\gamma[0,\tau]}$, this proves the theorem.
\end{proof}

\begin{proof}[Proof of (R2)$\Rightarrow$(R3')]
Let $(M,g_t)_{t\in I}$ be a smooth family of Riemannian manifolds such that the gradient estimate (R2) holds, and let $F:P_{(x,T)}\M\to \dR$ be a cylinder function.
Observe that
\begin{equation}\label{eq_observe_that}
|{\D^\p_\tau F}|(\gamma|_{[0,\tau]}\ast \gamma')=|\D_0^\p F_{\gamma[0,\tau]}|(\gamma').
\end{equation}
Now, using Theorem \ref{thm_quadrvar}, the gradient estimate (R2), and \eqref{eq_observe_that}, we compute (for a.e. $\gamma$ for a.e. $\tau$)
\begin{multline}\label{eq_proofmarting}
\sqrt{\frac{d[F^\bullet]_\tau}{d\tau}}(\gamma)=E_{(x,T)}\left[\sqrt{\frac{d[F^\bullet]_\tau}{d\tau}}\left|\Sigma^\tau\right.\right]=\sqrt{2}E_{(x,T)}\left[\abs{\D_y E_{(y,T-\tau)}F_{\gamma[0,\tau]}}(\pi_1\gamma_\tau)\left|\Sigma^\tau\right.\right]\\
\leq \sqrt{2}E_{(x,T)}\left[E_{\gamma_\tau}|\D_0^\p F_{\gamma[0,\tau]}|\left|\Sigma^\tau\right.\right]
=\sqrt{2}E_{(x,T)}\left[|{\D^\p_\tau F}|\left|\Sigma^\tau\right.\right],
\end{multline}
where we also used Proposition \ref{prop_condexp} in the last step. This proves (R3').
\end{proof}

\begin{proof}[Proof of (R3')$\Rightarrow$(R3)]
Let $F\in L^2(P_T\M,\Gamma_{(x,T)})$. Using the assumption (R3'), the Cauchy-Schwarz inequality, and the definition of the conditional expectation, we compute
\begin{equation}\label{eq_r3pr3}
E_{(x,T)}\frac{d[F^\bullet]_\tau}{d\tau}
\leq 2 E_{(x,T)}\left(E_{(x,T)}\left[|{\D^\p_\tau F}|\left|\Sigma^\tau\right.\right]\right)^2\leq 2 E_{(x,T)}|{\D^\p_\tau F}|^2.
\end{equation}
This proves the martingale estimate (R3).
\end{proof}

\subsection{Log-Sobolev inequality and spectral gap}

In this section, we prove the implications (R3')$\Rightarrow$(R4)$\Rightarrow$(R5).

\begin{proof}[Proof of (R3')$\Rightarrow$(R4)]
Let $F:P_T\M\to \dR$ be a cylinder function, and let $\{G^\tau\}_{\tau\in [0,T]}$ be the martingale induced by the function $G=F^2$, i.e.
$G^\tau=E_{(x,T)}[F^2|\Sigma^\tau]$.
Using the Ito formula and the martingale property we compute
\begin{equation}\label{logsob1}
 E_{(x,T)}[G^{\tau_2}\log G^{\tau_2}-G^{\tau_1}\log G^{\tau_1}]=E_{(x,T)}\int_{\tau_1}^{\tau_2} d (G^\tau\log G^\tau)=E_{(x,T)}\int_{\tau_1}^{\tau_2} \frac{1}{2G^\tau}\frac{d[G^\bullet]_\tau}{d\tau}d\tau.
\end{equation}
By assumption (R3'), the Cauchy-Schwarz inequality, and the definition of $G^\tau$, we have the estimate
\begin{equation}\label{logsob2}
 \frac{d[G^\bullet]_\tau}{d\tau}\leq 2 \left(E_{(x,T)}\left[|{2F \D^\p_\tau F}|\left|\Sigma^\tau\right.\right]\right)^2\leq 8 G^\tau\, E_{(x,T)}\left[|{\D^\p_\tau F}|^2\left|\Sigma^\tau\right.\right].
\end{equation}
Combining \eqref{logsob1} and \eqref{logsob2} we conclude that
\begin{equation}
 E_{(x,T)}[G^{\tau_2}\log G^{\tau_2}-G^{\tau_1}\log G^{\tau_1}]\leq 4 E_{(x,T)}\int_{\tau_1}^{\tau_2} E_{(x,T)}\left[|{\D^\p_\tau F}|^2\left|\Sigma^\tau\right.\right]\, d\tau=4E_{(x,T)}\langle F,\mathcal{L}_{\tau_1,\tau_2}F\rangle,
\end{equation}
where we used Propositon \ref{prop_malliavin} in the last step. This proves the log-Sobolev inequality (R4).
\end{proof}

\begin{proof}[Proof of (R4)$\Rightarrow$(R5)] Applying the log-Sobolev inequality for $F^2=1+\varepsilon G$ and using approximation, we obtain
\begin{equation}
 E_{(x,T)}[(G^{\tau_2})^2-(G^{\tau_1})^2]\leq 2E_{(x,T)}\langle G,\mathcal{L}_{\tau_1,\tau_2}G\rangle.
\end{equation}
Observing that $E_{(x,T)}[(G^{\tau_2})^2-(G^{\tau_1})^2]=E_{(x,T)}[(G^{\tau_2}-G^{\tau_1})^2]$, this proves the spectral gap.
\end{proof}


\subsection{Conclusion of the argument}

The goal of this final section is to prove the remaining implications (R5)$\Rightarrow$(R3)$\Rightarrow$(R2')$\Rightarrow$(R1).

\begin{proof}[Proof of (R5)$\Rightarrow$(R3)] Using the formula for the Malliavin gradient (Proposition \ref{prop_malliavin}) we can rewrite the spectral gap estimate (R5) in the form
\begin{equation}
E_{(x,T)}(F^{\tau_2}-F^{\tau_1})^2\leq 2 E_{(x,T)}\int_{\tau_1}^{\tau_2}\abs{\nabla_{\tau}^\p F}^2 d\tau.
\end{equation}
Dividing both sides by $\tau_2-\tau_1$ and limiting $\tau_2\to\tau_1$ we obtain
\begin{equation}
E_{(x,T)} \frac{d[F^\bullet]_\tau}{d\tau} \leq 2 E_{(x,T)}\abs{\nabla_{\tau}^\p F}^2,
\end{equation}
which is exactly the martingale estimate (R3).
\end{proof}

\begin{proof}[Proof of (R3)$\Rightarrow$(R2')] 
The quadratic variation formula (Theorem \ref{thm_quadrvar}) at $\tau=0$ reads
\begin{equation}
\abs{\nabla_x E_{(x,T)}F}^2=\tfrac12 E_{(x,T)} \frac{d[F^\bullet]_\tau}{d\tau}|_{\tau=0}.
\end{equation}
Together with the martingale estimate (R3) at $\tau=0$ this implies
\begin{equation}
 \abs{\nabla_x E_{(x,T)}F}^2 \leq E_{(x,T)}\abs{\nabla^\p F}^2,
\end{equation}
which is exactly the gradient estimate (R2').
\end{proof}

\begin{proof}[Proof of (R2')$\Rightarrow$(R1)] 
Let $(M,g_t)_{t\in I}$ be an evolving family of Riemannian manifolds satisfying the gradient estimate (R2'). 
Plugging in a $1$-point cylinder function $F=u\circ e_\sigma:P_T\M\to M\to \dR$, the estimate (R2') reduces to the estimate
\begin{equation}
\abs{\D P_{sT} u}^2\leq P_{sT}\abs{\D u}^2,
\end{equation}
c.f. Remark \ref{rem_gradest}. Thus, by Theorem \ref{thm_supersol} (only the implication (S3)$\Rightarrow$(S1) is needed), 
 $(M,g_t)_{t\in I}$ is a supersolution of the Ricci flow.
 To show that  $(M,g_t)_{t\in I}$  is also a subsolution, we will analyze the gradient estimate (R2') for a carefully chosen family of 2-point cylinder functions. Namely, given a point $(x,T)\in\mathcal{M}$ in space-time ($T>0$) and a unit tangent vector $v\in(T_xM,g_T)$ we choose a test function $u:M\times M\to\dR$ such that
\begin{equation}
 \textrm{grad}_{g_T}^{(1)}u=2v, \qquad \textrm{grad}_{g_T}^{(2)} u=- v,\qquad \Hess_{g_T} u=0\qquad\qquad \textrm{at} \,\, (x,x).
\end{equation}
We consider the 1-parameter family of test functions
\begin{equation}
 F^\sigma(\gamma)=u(e_0(\gamma),e_\sigma(\gamma)),
\end{equation}
where $\sigma\in[0,T]$. We will now analyze the asymptotics for $\sigma\to 0$. We start with the rough estimate
\begin{equation}\label{eq_roughassym}
 E_{(x,T)}\abs{\nabla^\p F^\sigma -v} = O(\sigma).
\end{equation}
Together with the gradient formula (Theorem \ref{thm_gradient_formula}) this implies that
\begin{equation}\label{eq_limitisone}
 \lim_{\sigma\to 0}\abs{\textrm{grad}_{g_T} E_{(x,T)}F^\sigma}^2=1=\lim_{\sigma\to 0} E_{(x,T)}\abs{\nabla^\p F^\sigma}^2.
\end{equation}
To compute the next order term, we first note that the gradient formula (Theorem \ref{thm_gradient_formula}) yields the estimate
\begin{equation}\label{eq_asym1}
\grad_{g_T} E_{(x,T)}F^\sigma=E_{(x,T)}[\D^\p F^\sigma]+\sigma(\Ric+\tfrac12\partial_t g)_{(x,T)}(v)+o(\sigma).
\end{equation}
Using this, we compute
\begin{align}\label{eq_asym2}
\tfrac{1}{2} \tfrac{d}{d\sigma}|_{\sigma=0}\left(\left|\grad_{g_T} E_{(x,T)}F^\sigma\right|^2-E_{(x,T)}\abs{\D^\p F^\sigma}^2\right)
&=\left\langle v, \tfrac{d}{d\sigma}|_{\sigma=0}\left(\grad_{g_T} E_{(x,T)}F^\sigma-E_{(x,T)}[\D^\p F^\sigma]\right)\right\rangle\nonumber\\
&=(\Ric+\tfrac12\partial_t g)_{(x,T)}(v,v).
\end{align}
Together with \eqref{eq_limitisone}, since the gradient estimate (R2') holds by assumption, we conclude that 
\begin{equation}
(\Ric+\tfrac12\partial_t g)_{(x,T)}(v,v)\leq 0.
\end{equation}
Since $(x,T)$ and $v$ are arbitrary, this proves that  $(M,g_t)_{t\in I}$ is a subsolution of the Ricci flow. Recalling that we already know that  $(M,g_t)_{t\in I}$ is a supersolution of the Ricci flow, this finishes the proof.
\end{proof}

\appendix

\section{A variant of Driver's integration by parts formula}

The purpose of this appendix is to prove Theorem \ref{thm_intbyparts}, a variant of Driver's integration by parts formula \cite{Driver_ibp}.
We write $(F,G)=E_{(x,T)}FG$. Moreover, if $v_\tau\in T_xM$ we use the notation $\langle v_\tau,dW_\tau\rangle=(U_0^{-1}v_\tau)_i\,  dW_\tau^i$. 

\begin{theorem}[Integration by parts]\label{thm_intbyparts}
Let $F,G:P_T\M\to\dR$ be cylinder functions, let $\{v_\tau\}_{\tau\in [0,T]}\in\cH$, and write $V=\{P_\tau^{-1}v_\tau\}_{\tau\in[0,T]}$. Then
\begin{equation}
D_V^\ast G=-D_V G + \tfrac12 G\!\! \int_0^T\!\!\langle  \tfrac{d}{d\tau} v_\tau-P_{\tau}(\Ric+\tfrac{1}{2}\partial_t{g})P_\tau^{-1} v_\tau,dW_\tau\rangle
\end{equation}
satisfies $(D_V F,G)=(F,D_V^\ast G)$. 
\end{theorem}

\begin{proof}
We adapt the proof from \cite[Sec. 8]{Hsu} to our setting of evolving manifolds.

Since $D_V$ satisfies the product rule it is enough to show that 
\begin{equation}\label{eq_to_prove}
E_{(x,T)}[D_V F]=\tfrac12 E_{(x,T)}[F\!\! \int_0^T\!\!\langle  \tfrac{d}{d\tau} v_\tau-P_{\tau}(\Ric+\tfrac{1}{2}\partial_t g)P_\tau^{-1} v_\tau,dW_\tau\rangle]
\end{equation}
for all cylinder functions $F$. We prove this by induction on the order $k$ of the cylinder function $F$.\\

\noindent{\underline{$k=1$:}} Let $F=e_\sigma^\ast u$ be a 1-point cylinder function, and let $s=T-\sigma$. Since $w(x,t):=P_{st}u(x)$ satisfies the heat equation, its gradient satisfies the equation
\begin{equation}
\nabla_t\, \grad_{g_t}\!w=\Lap_{g_t}\,\grad_{g_t}\!w-(\Ric+\tfrac12 \partial_t g)(\grad_{g_t}\!w,\cdot)^{\sharp_{g_t}},
\end{equation}
c.f. the proof of Proposition \ref{prop_evol_grad}. By the Feynman-Kac formula (Proposition \ref{prop_feynman_kac}) we have
\begin{equation}\label{app_gradient_formula}
\grad_{g_T}\!w\,(x,T)=E_{(x,T)}[R_\sigma P_\sigma \grad_{g_s}\!u\, (X_\sigma)],
\end{equation}
where $R_\tau=R_\tau(\gamma):(T_xM,g_T)\to (T_xM,g_T)$ solves the ODE $\tfrac{d}{d\tau} R_\tau=R_{\tau}P_\tau (\Ric+\tfrac12 \partial_t{g})_{T-\tau}P_\tau^{-1}$ with $R_0=\textrm{id}$, and where we view $(\Ric+\tfrac12 \partial_tg)_{T-\tau}$ as endomorphism of $TM$ (using the metric $g_{T-\tau}$).

By equation \eqref{repform_integrand} we have
\begin{align}\label{repform_integrand_app}
u(X_{\sigma})=w(x,T)+\int_0^{\sigma} \nabla^H \tilde{w}\,(U_\tau)\cdot dW_\tau,
\end{align}
where $\tilde{w}$ is the invariant lift of $w$ and $\nabla^H \tilde{w}=(H_1\tilde{w},\ldots,H_n\tilde{w})$ is its  horizontal gradient.

Let $\{z_\tau\}_{\tau\in [0,T]}\in \cH$. Using the above and the Ito isometry, we compute the following expectation value:
\begin{align}
E_{(x,T)}\,\, u(X_\sigma)\!\!\int_0^\sigma\! \langle R_\tau^\dag \dot{z}_\tau,dW_\tau\rangle
&=E_{(x,T)}\int_0^{\sigma}\! \nabla^H \tilde{w}\,(U_\tau)\cdot dW_\tau \int_0^\sigma \langle R_\tau^\dag \dot{z}_\tau,dW_\tau\rangle\\
&=2E_{(x,T)}\int_0^{\sigma}\! \langle  R_\tau^\dag \dot{z}_\tau, \nabla^H\tilde{w}(U_\tau)\rangle \, d\tau\\
&=2E_{(x,T)}\int_0^{\sigma}\! \langle \dot{z}_\tau, R_\tau U_0\nabla^H\tilde{w}(U_\tau)\rangle_{g_T} \, d\tau.
\end{align}
Let $N_\tau :=R_\tau U_0\nabla^H\tilde{w}(U_\tau)=R_\tau P_\tau\,\grad_{g_{T-\tau}}\!\!\!\! w(X_\tau,T-\tau)$. Integration by parts gives
\begin{equation}
E_{(x,T)}\int_0^{\sigma}\! \langle \dot{z}_\tau, N_\tau\rangle_{g_T} \, d\tau=E_{(x,T)}[\langle z_\sigma, N_\sigma\rangle_{g_T}-\int_0^{\sigma}\! \langle {z}_\tau, dN_\tau\rangle_{g_T}]= E_{(x,T)}\langle z_\sigma, N_\sigma\rangle_{g_T},
\end{equation}
where in the last step we used that $N_\tau$ is a martingale, c.f. equation \eqref{eq_proof_of_feynkac}. Putting things together, and taking also into account that
\begin{equation}
E_{(x,T)}\,\, u(X_\sigma)\!\!\int_\sigma^T\! \langle R_\tau^\dag \dot{z}_\tau,dW_\tau\rangle
=E_{(x,T)} u(X_\sigma)\,\,E_{(x,T)}\!\!\int_\sigma^T\! \langle R_\tau^\dag \dot{z}_\tau,dW_\tau\rangle=0,
\end{equation}
we obtain
\begin{equation}
E_{(x,T)}\,\, u(X_\sigma)\!\!\int_0^T\! \langle R_\tau^\dag \dot{z}_\tau,dW_\tau\rangle
=2E_{(x,T)}\langle R_\sigma^\dag z_\sigma, P_\sigma\grad_{g_s}\! u(X_\sigma) \rangle_{g_T}.
\end{equation}
Finally, we let $v_\tau=R_\tau^\dag z_\tau$. Then
\begin{equation}
R_\tau^\dag \dot{z}_\tau=\dot{v}_\tau-P_\tau(\Ric+\tfrac{1}{2}\partial_t g)P_\tau^{-1}v_{\tau},
\end{equation}
and equation \eqref{eq_to_prove} follows.\\

\noindent{\underline{$k-1\rightarrow k$:}} Let $F = e_\sigma^\ast f$ be a k-point cylinder function and let $s_i = T -\sigma_i$.
Define a new function of $k-1$ variables by
\begin{equation}
g(x_1,\ldots, x_{k-1})=E_{(x_{k-1},s_{k-1})}f(x_1,\ldots,x_{k-1},X'_{\sigma_k-\sigma_{k-1}}),
\end{equation}
where $X'$ is based at $x_{k-1}$. Let $G:P_{(x,T)}\M\to\dR$ be the $(k-1)$-point cylinder function
\begin{equation}
G(\gamma)=g(e_{\sigma_1}\gamma,\ldots,e_{\sigma_{k-1}}\gamma).
\end{equation}
In belows computation we will frequently use the Markov property (Proposition \ref{prop_condexp}). 

The first step is to express
\begin{equation}
E_{(x,T)}D_VF=\sum_{j=1}^k E_{(x,T)}\langle v_{\sigma_j},P_{\sigma_j}\grad_{g_{s_j}}^{(j)} f(X_{\sigma_1},\ldots,X_{\sigma_k})\rangle_{g_T}
\end{equation}
in terms of $G$.
To this end, note that for $j=1,\ldots, k-2$ we simply have
\begin{equation}
\grad^{(j)}_{g_{s_j}}g(x_1,\ldots,x_{k-1})=E_{(x_{k-1},s_{k-1})}\grad^{(j)}_{g_{s_j}}f(x_1,\ldots,x_{k-1},X'_{\sigma_k-\sigma_{k-1}}).
\end{equation}
For $j=k-1$ using the product rule and the gradient formula \eqref{app_gradient_formula} we have
\begin{multline}
\grad^{(k-1)}_{g_{s_{k-1}}}g(x_1,\ldots,x_{k-1})=
E_{(x_{k-1},s_{k-1})}\grad^{(k-1)}_{g_{s_{k-1}}}f(x_1,\ldots,x_{k-1},X'_{\sigma_k-\sigma_{k-1}})\\
+E_{(x_{k-1},s_{k-1})} R'_{\sigma_k-\sigma_{k-1}} P'_{\sigma_k-\sigma_{k-1}}  \grad_{g_{s_k}}\!f\, (x_1,\ldots,x_{k-1},X'_{\sigma_k-\sigma_{k-1}}),
\end{multline}
where $R'_\tau=R'_\tau(\gamma)\!:\!(T_{x_{k-1}}M,g_{s_{k-1}})\to (T_{x_{k-1}}M,g_{s_{k-1}})$ solves the ODE
$\tfrac{d}{d\tau} R'_\tau=R'_{\tau}P'_\tau (\Ric+\tfrac12 \partial_t{g})_{s_{k-1}-\tau}P'^{-1}_\tau$ with $R_0=\textrm{id}$.
Taking expectations, we thus obtain
\begin{multline}
E_{(x,T)}D_VF=E_{(x,T)}D_VG+E_{(x,T)}\langle v_{\sigma_k},P_{\sigma_k}\grad_{g_{s_k}}^{(k)} f(X_{\sigma_1},\ldots,X_{\sigma_k})\rangle_{g_T}\\
-E_{(x,T)}E_{(X_{\sigma_{k-1}},s_{k-1})}\langle v_{\sigma_{k-1}},P_{\sigma_{k-1}} R'_{\sigma_k-\sigma_{k-1}} P'_{\sigma_k-\sigma_{k-1}}  \grad_{g_{s_k}}^{(k)}\!f\, (X_{\sigma_1},\ldots,X_{\sigma_{k-1}},X'_{\sigma_k-\sigma_{k-1}})\rangle_{g_T}.
\end{multline}
By the induction hypothesis we have
\begin{equation}\label{app_term1}
E_{(x,T)}D_VG=\tfrac12 E_{(x,T)}[G \int_0^{\sigma_{k-1}}\langle  \tfrac{d}{d\tau} v_\tau-P_{\tau}(\Ric+\tfrac{1}{2}\partial_t g)P_\tau^{-1} v_\tau,dW_\tau\rangle].
\end{equation}
Conditioning, using the induction hypothesis for $1$-point functions, and unconditioning again, we compute
\begin{align}\label{app_term2}
&\!\!\!\!\!\!\!\!\!\!\!\!\!\!\!\!\!\!E_{(x,T)}\langle v_{\sigma_k}-v_{\sigma_{k-1}},P_{\sigma_k}\grad_{g_{s_k}}^{(k)} f(X_{\sigma_1},\ldots,X_{\sigma_k})\rangle_{g_T}\nonumber\\
&=E_{(x,T)}E_{(X_{\sigma_{k-1}},s_{k-1})}\langle P_{\sigma_{k-1}}^{-1}( v_{\sigma_k}-v_{\sigma_{k-1}}),P'_{\sigma_k-\sigma_{k-1}}\grad_{g_{s_k}}^{(k)} f(X_{\sigma_1},\ldots,X_{\sigma_{k-1}},X'_{\sigma_k})\rangle_{g_{s_{k-1}}}\nonumber\\
&=\tfrac12 E_{(x,T)}[F \int_{s_{k-1}}^T\langle  \tfrac{d}{d\tau} v_\tau-P_{\tau}(\Ric+\tfrac{1}{2}\partial_t g)P_\tau^{-1} (v_\tau-v_{\sigma_{k-1}}),dW_\tau\rangle].
\end{align}
Finally, using the induction hypothesis for $1$-point functions and the ODE for $R'$ we compute
\begin{align}\label{app_term3}
&\!\!\!\!\!\!E_{(x,T)}E_{(X_{\sigma_{k-1}},s_{k-1})}\langle v_{\sigma_{k-1}},(P_{\sigma_k}-P_{\sigma_{k-1}} R'_{\sigma_k-\sigma_{k-1}} P'_{\sigma_k-\sigma_{k-1}})  \grad_{g_{s_k}}^{(k)}\!f\, (X_{\sigma_1},\ldots,X_{\sigma_{k-1}},X'_{\sigma_k-\sigma_{k-1}})\rangle_{g_T}\nonumber\\
&=E_{(x,T)}E_{(X_{\sigma_{k-1}},s_{k-1})}\langle (I-R_{\sigma_{k}-\sigma_{k-1}}'^\dag) P_{\sigma_{k-1}}^{-1} v_{\sigma_{k-1}}, P'_{\sigma_k-\sigma_{k-1}}  \grad_{g_{s_k}}^{(k)}\!f\, (X_{\sigma_1},\ldots,X_{\sigma_{k-1}},X'_{\sigma_k-\sigma_{k-1}})\rangle_{g_{s_{k-1}}}\nonumber\\
&= E_{(x,T)}[F \int_{s_{k-1}}^{s_k}\langle  P_{\tau}(\Ric+\tfrac{1}{2}\partial_t g)P_\tau^{-1} v_{\sigma_{k-1}},dW_\tau\rangle].
\end{align}
Adding \eqref{app_term1}, \eqref{app_term2} and \eqref{app_term3} we conclude that
\begin{equation}
 E_{(x,T)}D_VF=\tfrac12 E_{(x,T)}[F\!\! \int_0^T\!\!\langle  \tfrac{d}{d\tau} v_\tau-P_{\tau}(\Ric+\tfrac{1}{2}\partial_t g)P_\tau^{-1} v_\tau,dW_\tau\rangle].
\end{equation}
This proves the theorem.
\end{proof}
\bibliography{HN_weakricciflow}

\bibliographystyle{alpha}

\vspace{10mm}
{\sc Robert Haslhofer, Courant Institute of Mathematical Sciences, New York University, 251 Mercer Street, New York, NY 10012, USA}\\

{\sc Aaron Naber, Department of Mathematics, Northwestern University, 2033 Sheridan Road, Evanston, IL 60208, USA}\\

\emph{E-mail:} robert.haslhofer@cims.nyu.edu, anaber@math.northwestern.edu

\end{document}